\documentclass[a4paper,11pt,fleqn]{article}
\usepackage[official]{eurosym}
\usepackage{amsmath}
\usepackage{amssymb}
\usepackage{theorem}
\usepackage[usenames,dvipsnames]{pstricks}
\usepackage{euscript}
\usepackage{graphicx}
\usepackage{charter}
\usepackage{pifont}
\usepackage{comment}
\usepackage{pgfplots,upgreek,subfigure}

\newtheorem{theorem}{Theorem}[section]
\newtheorem{lemma}[theorem]{Lemma}
\newtheorem{corollary}[theorem]{Corollary}
\newtheorem{proposition}[theorem]{Proposition}
\theoremstyle{plain}{\theorembodyfont{\rmfamily}%
\newtheorem{definition}[theorem]{Definition}}

\RequirePackage[colorlinks,hyperindex]{hyperref} 
\hypersetup{linktocpage=true,citecolor=dblue,linkcolor=dgreen}
\usepackage{cleveref}
\PassOptionsToPackage{normalem}{ulem}
\usepackage{ulem}
\renewcommand{\leq}{\ensuremath{\leqslant}}
\renewcommand{\geq}{\ensuremath{\geqslant}}
\renewcommand{\le}{\ensuremath{\leqslant}}
\renewcommand{\ge}{\ensuremath{\geqslant}}

\newcommand{\email}{}

\input{alphabet}

\definecolor{labelkey}{rgb}{0,0.08,0.45}
\definecolor{refkey}{rgb}{0,0.6,0.0}
\definecolor{Brown}{rgb}{0.45,0.0,0.05}
\definecolor{dgreen}{rgb}{0.00,0.49,0.00}
\definecolor{dblue}{rgb}{0,0.08,0.75}

\usepackage{amsopn}

\newcommand{\zer}{\ensuremath{\operatorname{zer}}}

\newcommand{\dist}{\operatorname{dist}}

\theoremstyle{plain}{\theorembodyfont{\rmfamily}%
}
\theoremstyle{plain}{\theorembodyfont{\rmfamily}%
\newtheorem{assumption}[theorem]{Assumption}}
\theoremstyle{plain}{\theorembodyfont{\rmfamily}%
\theoremstyle{plain}{\theorembodyfont{\rmfamily}%
}
\theoremstyle{plain}{\theorembodyfont{\rmfamily}%
\newtheorem{example}[theorem]{Example}}
\theoremstyle{plain}{\theorembodyfont{\rmfamily}%
\newtheorem{remark}[theorem]{Remark}}
\theoremstyle{plain}{\theorembodyfont{\rmfamily}%
\theoremstyle{plain}{\theorembodyfont{\rmfamily}%
}

\numberwithin{equation}{section}

\begin{document}

\title{\sffamily A stochastic use of the Kurdyka-\L ojasiewicz property: \\ Investigation of optimization algorithms behaviours
in a non-convex differentiable framework.  
}

\author{
Jean-Baptiste Fest
\thanks{\textbf{Corresponding author},  CMAP, Ecole Polytechnique, Institut Polytechnique de Paris, Palaiseau, France \\ (\email{jean-baptiste.fest@polytechnique.edu}).}
\and Audrey Repetti\thanks{
Maxwell Institute for Mathematical Sciences and School of Mathematical and Computer Sciences, Heriot-Watt University, Edinburgh, UK (\email{a.repetti@hw.ac.uk}).} 
\and 
Emilie Chouzenoux\thanks{CVN, CentraleSupélec, Inria, Université Paris-Saclay, Gif-sur-Yvette, France (\email{emilie.chouzenoux@inria.fr}).}
\and
}

\date{}

\maketitle

\vskip 8mm

\maketitle

\begin{abstract}
Asymptotic analysis of generic stochastic algorithms often relies on descent conditions. 
In a convex setting, some technical shortcuts can be considered to establish asymptotic convergence guarantees of the associated scheme. However, in a non-convex setting, obtaining similar guarantees is usually more complicated, and relies on the use of the Kurdyka-\L ojasiewicz (K\L) property. 
While this tool has become popular in the field of deterministic optimisation, it is much less widespread in the stochastic context and the few works making use of it are essentially based on \textit{trajectory-by-trajectory} approaches. 
In this paper, we propose a new framework for using the $\text{K\L}$ property in a non-convex stochastic setting based on conditioning theory. We show that this framework allows for deeper asymptotic investigations on stochastic schemes verifying some generic descent conditions. 
We further show that our methodology can be used to prove convergence of generic stochastic gradient descent (SGD) schemes, and unifies conditions investigated in multiple articles of the literature. 
\end{abstract}

Stochastic processes, non-convex optimization, Kurdyka-\L ojasiewicz property, stochastic gradient descent. 

65K05, 
90C26, 
90C15, 
90C53. 

\section{Introduction}

The objective of this work is to approximate a solution of an unconstrained optimization problem of the form of 
\begin{equation} \label{eq:pb}
\underset{\wb\in \eR^N}{\text{minimize}}\quad F(\wb),
\end{equation}
where $F\colon \eR^N \to \eR$ is a continuously differentiable function.  
Specifically, we will investigate the behaviour of a process $(\wv_k)_{k\in \Nbb}$ defined on a probabilistic space $\left(\Omega, \Fc, \Pbb \right)$ and belonging to the finite dimensional space $\eR^N$, that will aim to asymptotically approximate such a solution
\cite{duflo2013random, robbins1951stochastic}. 
Celebrated examples are often based on stochastic gradient descent (SGD) algorithms. 

When \eqref{eq:pb} is solved using a deterministic scheme, a set of recent proofs of convergence results rely on the Kurdyka-\L ojasiewicz ($\text{K\L}$) property, that has the advantage of promoting interesting asymptotic behaviors when $F$ is non-necessarily convex. 
In this context, $\text{K\L}$   has been used to prove convergence of proximal point algorithms \cite{attouch2010proximal, attouch2013convergence}, of simple splitting algorithms such as the forward-backward algorithm and its variants~\cite{attouch2013convergence, chouzenoux2014variable, Bolte2014, frankel2015, chouzenoux2016, repetti2021,bonettini2020convergence,bonettini2021new}, as well as other algorithms based on the majorization-minimization principle \cite{chouzenoux2013majorize,chouzenouxconvergence,chalvidal2022block}.
A natural question is to investigate the transfer of such proof techniques from the deterministic setting, to the stochastic setting, for asymptotic analysis including a.s. convergence of stochastic processes.
Such an extension is quite challenging, mainly due to the dynamics of the functions involved in $\text{K\L}$   conditions, that cannot be controlled in a stochastic environment.

In this work we will develop a new framework, based on $\text{K\L}$   theory \cite{kurdyka1998gradients, bolte2008characterizations}, to derive almost sure (a.s.) convergence guarantees of the stochastic process $(\wv_k)_{k\in \Nbb}$, for the resolution of \eqref{eq:pb} when $F$ is not convex.

\paragraph{Stochastic approximation theory.}
Following stochastic approximation theory introduced in \cite{robbins1951stochastic,robbins1971convergence} and similarly to the stabilisation theory for dynamic systems \cite{terrell2009stability}, the asymptotic behaviour of $(\wv_k)_{k\in \Nbb}$ to approximate the minimum of $F$ can be investigated by introducing an auxiliary function $V\colon\Hc \to \Rbb$ acting as a Lyapunov function, with $\Hc$ being a finite-dimensional Euclidean space. 
To this aim, we introduce an auxiliary augmented process $(\xv_k)_{k\in \Nbb}$ defined on $\left(\Omega, \Fc, \Pbb \right)$ that will be used to study the behaviour of $V$, including process $(\wv_k)_{k\in \Nbb}$ and any underlying parameters of the associated recursive scheme (e.g., step-sizes). 
In particular, similarly to \cite{williams1991probability}, in this work we will focus on functions $V$ such that $\left(V(\xv_k)\right)_{k\in \Nbb}$ follows a supermartingale behavior. 
Such an approach is instrumental to avoid the strong assumption that process $(F(\wv_k))_{k\in \Nbb}$ verifies a supermartingale condition. In fact, stochastic approximation theory results highlight that $(F(\wv_k))_{k\in \Nbb}$ usually only verifies almost-supermartingale condition \cite{robbins1971convergence}. In many cases, however, it is fairly easy to construct $V$ that holds supermartingale condition and enabling to study the asymptotic behaviour of $(F(\wv_k))_{k\in \Nbb}$ \cite{robbins1971convergence}. \\
In this work, we assume that such a Lyapunov function $V$ exists, and we aim to design a framework, with mild assumptions, that allows obtaining stronger asymptotic convergence results on $(\wv_k)_{k\in \Nbb}$, including its almost-sure convergence to a solution to problem~\eqref{eq:pb}. 
Such a result is challenging to establish in a non-convex stochastic context, and has mainly been studied so far for particular cases. 
For instance, \cite{ljung1978strong} proved this result when the set of accumulation points of $(\wv_k)_{k\in \Nbb}$ is included in a subset that is only made up of isolated points. Without this strong assumption, the most common approach to handle non-convex problems consists in adopting the ODE method \cite{benaim1996dynamical}. This method uses the fact that the asymptotic behaviour of most stochastic schemes is very similar to the one of some differential equations. 
Nevertheless the noise conditions considered in these works may be difficult to verify in practice, hence restricting their practical use.
Instead, in this work we aim to introduce a framework that allows the use of the $\text{K\L}$   property to investigate almost-sure convergence of $(\wv_k)_{k\in \Nbb}$ under mild assumptions that are generally assumed to be true in modern stochastic optimization (i.e., supermartingale like properties).

However, as explained earlier, although $\text{K\L}$   conditions have been widely used to investigate convergence of deterministic schemes in a non-convex context, their use in a stochastic framework is not straightforward. 
In the next paragraphs we better explain the \L ojasiewicz and $\text{K\L}$   properties, and give an overview of their use in the stochastic literature.

\paragraph{Non-convex optimization and \L ojasiewicz theory.} 
\L ojasiewicz condition \cite{lojasiewicz1993geometrie} was originally introduced to study the behaviour of trajectories of some bounded gradient flow in a continuous setting, for the class of analytical functions, under the form of a local curvature relation. 
The works of \L ojasiewicz have then been transferred later to the discrete deterministic framework. In particular, it is shown in \cite{absil2005convergence} that the \L ojasiewicz condition can be used to establish the convergence of a deterministic iterative approximation scheme starting from a descent condition. \\
The \L ojasiewicz condition being originally designed to study gradient flow trajectories, it naturally shares strong connections with the ODE method. 
Indeed, combining this fact to works such as \cite{benaim1996dynamical, chill2009applications}, Benaim provided in \cite{benaim2016strict} strong convergence guarantees of the SGD algorithm when applied to real analytic non-convex cost functions, without further assumption on the set of accumulating points of $(\wv_k)_{k\in \Nbb}$. 
Almost at the same time, Tadi\'{c} published similar results in \cite{tadic2015convergence}, proving convergence of SGD iterations under the \L ojasiewicz condition for real analytic non-convex cost functions, using a trajectory-based proof.
Although the works of Benaim and Tadi\'{c} \cite{benaim1996dynamical, benaim2016strict, tadic2015convergence} based on \L ojasiewicz condition 
made a breakthrough to investigating asymptotic behaviour of stochastic processes within a non-convex context, the ``trajectory-by-trajectory approach" of their analysis restricts their framework to noise conditions that may be difficult to satisfy in practice. 
%
%
In the recent work \cite{dereich2021convergence}, the noise conditions have been relaxed thanks to the use of a combination of conditioning theory with the use of \L ojasiewicz condition in a non-convex setting. 


\paragraph{Non-convex optimization and $\text{K\L}$ theory.} 
Kurdyka proposed a generalization of the \L ojasiewicz condition in \cite{kurdyka1998gradients} that holds for differentiable functions definable an o-minimal structure. Similarly to the \L ojasiewicz condition, the use of $\text{K\L}$   condition for discrete deterministic approximations is fairly recent, and its popularity is mainly due to the pioneering works \cite{bolte2010characterizations, attouch2010proximal, Bolte2014}. 
Nevertheless, the use of $\text{K\L}$ property in the stochastic framework remains rare and seems to be limited to a few articles. 
In \cite{gadat2017optimal}, the authors used the $\text{K\L}$   condition to show $L^1$ and $L^2$ convergence of some stochastic schemes.
Further, \cite{driggs2020spring} proposed a new version of the $\text{K\L}$   condition, in expectation, to study the a.s. convergence of their specific stochastic proximal-gradient algorithm leveraging a mini-batch structure, in a non-convex non-differentiable setting. 
In \cite{milzarek2023convergence} the authors considered a similar algorithm with more general noise conditions, but with a more restrictive $\text{K\L}$   formulation and using a ``trajectory-by-trajectory" approach.  Finally, the authors of \cite{li2021convergence} mention that a stochastic formulation of the $\text{K\L}$   inequality would enable to study the accumulation points of their process, without however providing any theoretical results. \\
Hence, to our knowledge, no asymptotic analysis including a.s. convergence of generic stochastic processes has been developed yet.

\paragraph{Contributions.}
In this paper, we develop a framework to study the a.s. convergence of a generic stochastic process $(\wv_k)_{k\in \Nbb}$ to a critical point of a differentiable non-convex function $F$, under the $\text{K\L}$   condition.
Starting from an augmented process $\left(V(\xv_k)\right)_{k\in \Nbb}$ following a supermartingale behaviour, we make use of the $\text{K\L}$   condition (instead of typical convexity assumption) to obtain strong asymptotic results. 
Unlike the ODE method or the ``trajectory-by-trajectory" approach, our strategy is reminiscent of the deterministic proof methodology typically developed in \cite{Bolte2014}, that we adapt to a stochastic framework based on conditioning theory. 
We further show that our conditions enable the convergence of a wide class of state-of-the-art SGD algorithms \cite{berahas2016multi, bertsekas2000gradient, bollapragada2019exact, bottou2018optimization, chouzenoux2022sabrina, ghadimi2013stochastic, gower2019sgd, jahani2021doubly, schmidt2013fast, wang2017stochastic}, summarized in Table~\ref{tab1}. 

\paragraph{Article outline.} 
The rest of the paper is organized as follows. 
Section~\ref{sec2} presents the basic assumptions and foundational results needed throughout the article. 
Section~\ref{sec3} can be considered as the cornerstone of our theoretical analysis. In this section, we introduce a framework to enable the use of the $\text{K\L}$   property (for non-convex functions) in a stochastic setting. 
We then leverage this framework in Section~\ref{sec4}, to present our main asymptotical result, consisting in a summability condition involving a generic (differentiable, non-convex) Lyapunov function $V$. 
In Sections~\ref{sec5} and~\ref{sec6} we apply these results to deduce convergence guarantees of the SGD scheme, under generic assumptions. 
Finally, we give our conclusion in Section~\ref{sec7} and discuss the possible extension of the current work to the non-differentiable case. 




\paragraph{Notation}
$(\Hc,\left\langle \cdot | \cdot \right\rangle_{\Hc} )$ correspond to the (finite-dimensional) Euclidean space under study. 
$\|\cdot \|_{\Hc}$ denotes the canonical norm associated with $\Hc$. For any subset $E\subset \Hc$, the distance function to this set is denoted by $\dist_{\Hc}(\cdot, E) = \inf_{\xb\in \Hc} \| \cdot - \xb \|_{\Hc}$. 
Bold letters as $\xb$ are used for deterministic vectors, while straight bold letters as $\xv$ are used for stochastic vectors. Similarly, $x$ is used for denoting a deterministic scalar variable, while straight $\xD$ denotes a stochastic scalar variable. 
For a given function $F \colon \Hc\longrightarrow \Rbb$, the variable $F(\xv)$ denotes the stochastic value of function $F$ evaluated at the stochastic variable $\xv$. 
For a given $E\subset \Hc$ we denote $|E|$ the cardinal of $E$, and $F(E):=\{F(\xb)~|~\xb\in \Hc\}$ is the image under $F$ of $E$. Considering $U$ an open-set of $\Hc$, if $F$ is differentiable at $\xb \in U$, $\nabla F(\xb)$ denotes the gradient of $F$ at $\xb$. If $F$ is differentiable over $\Hc$, $\zer \nabla F$ corresponds to the set of zeros of $\nabla F$, i.e., the set of critical/stationary points of $F$. \\
Let $\left(\Omega, \Fc, \Pbb \right)$ be a probability space.
We say that a condition holds \emph{almost surely} (a.s.) if it holds on a probability-one event of $\mathcal{F}$.
We denote by $\Ebb[\cdot]$ the expectation operator, and $\Ebb[\cdot|\Gc]$ the conditional expectation operator with respect to a generic sub sigma-algebra $\Gc\subset \Fc$. 

\section{General assumptions and preliminary results}
\label{sec2}

%
Let $(\xv_k)_{k\in \Nbb}$ be a process belonging to $\left(\Omega, \Fc, \Pbb \right)$, and $\chi_{\infty}$ be its set of accumulation (or cluster) points\footnote{$\chi_{\infty}$ is a random set from $\Omega$ to $2^{\Hc}$, the set of subsets of $\Hc$.}. 
We introduce the assumptions and the framework under which we will study the behaviour of $(\xv_k)_{k\in \Nbb}$, focusing on a type-Lyapunov function $V \colon \Hc \to \Rbb$. 


\subsection{Assumptions}\label{P3.1}

The following generic assumptions will guide us throughout the remainder of this work.  

\begin{assumption}
\label{ass:H1}
$V$ is coercive and continuously differentiable. 
\end{assumption}

\begin{assumption}
\label{ass:H2}
\begin{enumerate}
    \item  \label{ass:H2:i}
    The sequence $(V(\xv_k))_{k\in \Nbb}$ converges a.s. to a random variable $\VD_{\infty}$, such that $\VD_{\infty}<+\infty$ a.s.
    \item \label{ass:H2:ii}
    There exists a deterministic non-empty subset $\Gamma$ of $\Hc$ such that $|V(\Gamma)|<+\infty$ a.s. and $\Gamma \cap \chi_\infty \neq \emptyset$. 
\end{enumerate}
\end{assumption}

A few remarks can be made on the generality of these assumptions:
\begin{remark}
    \begin{enumerate}
        \item Assumption~\ref{ass:H1} ensures the existence of a minimizer for $V$ \cite{bertsekas1997nonlinear}.
        \item Assumption~\ref{ass:H2} gives information on the behaviour of $(\xv_k)_{k\in \Nbb}$, and is thus similar to the general field of stochastic approximation \cite{duflo2013random}. 
        \item Assumption~\ref{ass:H2}\ref{ass:H2:i} can usually  be deduced from an approximation scheme verifying a descent condition of the form of an almost-supermatingale inequality \cite{robbins1971convergence}. 
        \item 
        Assumption~\ref{ass:H2}\ref{ass:H2:ii} will be necessary in the stochastic framework considered in this work, in particular due to the fact that $\chi_{\infty}$ is a random set. This assumption can be interpreted as the need for almost all trajectories $(\chi_{\infty}(\omega))_{\omega \in \Omega}$ to share a common asymptotic characteristic, represented by their common intersection with a finite set $\Gamma$. This set will be used as a substitute to the ``trajectory-by-trajectory" approach adopted in works such as \cite{benaim1996dynamical, benaim2016strict, tadic2015convergence}, in order to deduce more general convergence guarantees. Specifically, $\Gamma$ will be used to construct uniformization properties similar to the deterministic work \cite{Bolte2014}, that will allow us to work independently from $\omega$-variable. 
        It is worth noticing that Assumption~\ref{ass:H2}\ref{ass:H2:ii} is not too constraining, and such a $\Gamma$ exists for a wide class of problems. 
    \end{enumerate}
\end{remark}

The practicality of these conditions will be studied in the context of the SGD scheme, in Sections~\ref{sec5} and \ref{sec6}.

\subsection{Preliminary results}

The next two results will guide us throughout the rest of our analysis. The first one, Proposition~\ref{prop0}, will provide us with some indications on the structure of subset $\Gamma$ verifying Assumption~\ref{ass:H2}\ref{ass:H2:ii} in a generic context. The second result, Proposition~\ref{prop1}, gives some properties on process $(\xv_k)_{k\in \Nbb}$.

\begin{proposition}\label{prop0}
Let $f: \Hc \to \Rbb$ be a continuous coercive function and $C$ be a non-empty subset of $\Hc$ such that $|f(C)|<+\infty$. 
The closure of $C$, denoted $\overline{C}$, can be written as a finite union of non-empty compact sets, such that $f$ is constant on each of them. 
\end{proposition}

\vspace{0.4cm}

\begin{proof} 
We first show that $C$ is bounded. If not, there would exist a sequence $(\vb_k)_{k\in \Nbb}$ of vectors of $C$ such that $\|\vb_k\|_{\Hc}\underset{k\to +\infty}{\longrightarrow} +\infty$, which, by coercivity, conducts to $f(\vb_k)\underset{k\to +\infty}{\longrightarrow} +\infty$ and finally contradicts the finiteness of $f(C)$.

Since $|f(C)|<+\infty$, we can denote $I=|f(C)|$ and $f(C)=\left \{f_1, \ldots,f_I \right\}~(f_1\neq f_2 \cdots\neq f_I)$. 
As $C$ is bounded, it follows that $\overline{C}$ is also bounded. Moreover, since $f(C)$ is finite and $f$ continuous, we also have $f\left(\overline{C}\right)=f(C) $. 
So $\overline{C}\subset f^{-1}\left(f(C)\right)=\bigcup_{i=1}^I f^{-1}\left(\{f_i\} \right)$, and we can write $\overline{C}=\bigcup_{i=1}^I C_i$ where, for all $i\in \{1, \ldots,I\}$, $C_i=f^{-1}\left(\{f_i\} \right)\cap \overline{C}$.

Let us now fix $i\in \{1,\ldots,I\}$. 
On the one hand, $C_i$ is closed has an intersection of two closed sets. 
On the other hand, since $\overline{C}$ is bounded, so is $C_i$. 
We can hence deduce that $C_i$ is compact (since $\Hc$ is of finite dimension). 
Moreover, due to $f(\overline{C})=f(C)(=\{f_1,\ldots,f_I\})$, there exists $\vb\in \overline{C}$ such that $f(\vb)=f_i$ and so $C_i$ is not empty. 
Finally, the fact that $C_i\subset f^{-1}\left(\{f_i\} \right)$ ensures that $f$ is constant on $C_i$ (with $f(C_i)=f_i$).     
\end{proof}

The next proposition provides technical topological results on $\chi_{\infty}$, the set of cluster points of $(\xv_k)_{k\in \Nbb}$, and on $\VD_{\infty}$, the limit of $(V(\xv_k))_{k\in \Nbb}$. 
The proof of this result is reminiscent from classical arguments often encountered in the deterministic non-convex setting as in \cite{attouch2010proximal,Bolte2014}, that we generalize here to a stochastic framework.

\begin{proposition}\label{prop1}
Under Assumptions~\ref{ass:H1} and \ref{ass:H2}, we have 
\begin{enumerate}
    \item\label{prop1:i} $(\xv_k)_{k\in \Nbb}$ is a.s. bounded.
    \item\label{prop1:ii} $\chi_{\infty}$ is a.s. non empty and compact.
    \item\label{prop1:iii} $\left(\dist(\xv_k,\chi_{\infty})\right)_{k\in \Nbb}$ converges a.s. to $0$.
    \item\label{prop1:iv} $\VD_{\infty}(\omega)\in V(\Gamma)$ a.s..
\end{enumerate}
\end{proposition}

\vspace{0.4cm}

\begin{proof}

According to Assumption~\ref{ass:H2}, there exists a set $\Lambda\subset\Omega$ of probability one where, for all $\omega \in \Lambda$,
\begin{subequations}
\begin{align} 
\label{4.20a}
&\lim\limits_{k\rightarrow+\infty} V(\xv_k(\omega))=V_{\infty}(\omega)<+\infty, \\ 
\label{4.20b}
&\Gamma \cap \chi_\infty(\omega) \neq \varnothing.
\end{align}
\end{subequations}
Inequality \eqref{4.20a} implies that $(V(\xv_k(\omega)))_{k \in \eN}$ is a bounded sequence. 
According to Assumption~\ref{ass:H1}, $V$ being coercive, $(\xv_k(\omega))_{k \in \eN}$ is also bounded. 
It follows that the set of cluster points $\chi_{\infty}(\omega)$ is non empty, bounded and closed in $\Hc$, hence compact. This proves \ref{prop1:i} and \ref{prop1:ii}. 

We now show that \ref{prop1:iii} holds. 
By contradiction, if \ref{prop1:iii} does not hold, then, for every $\omega \in \Lambda$,
there exist $\varepsilon>0$ and a subsequence $(\xv_{\phi(k)}(\omega))_{k\in \Nbb}$ such that 
\begin{equation}\label{4.20c}
    (\forall k \in \Nbb)\quad \dist \left(\xv_{\phi(k)}(\omega),\chi_{\infty}(\omega)\right)>\epsilon.
\end{equation}
Since $(\xv_k(\omega))_{k\in \Nbb}$ is bounded, $(\xv_{\phi(k)}(\omega))_{k\in \Nbb}$ is also bounded. So the set of cluster points of $(\xv_{\phi(k)}(\omega))_{k\in \Nbb}$ is non-empty and included in $\chi_{\infty}(\omega)$.
Thus, there exists another subsequence $(\xv_{(\phi\circ \widetilde{\phi})(k)}(\omega))_{k\in \Nbb}$ and a point $\xb_{\infty} (\omega)\in \chi_{\infty}(\omega)$ such that $\left\|\xv_{(\phi\circ \widetilde{\phi})(k)}(\omega)- \xb_{\infty} (\omega)\right\|\underset{k\to+\infty}{\longrightarrow} 0$. Hence, $\dist \left(\xv_{(\phi\circ \widetilde{\phi})(k)}(\omega),\chi_{\infty}(\omega)\right)\underset{k\to+\infty}{\longrightarrow} 0$,
which contradicts \eqref{4.20c} and thus makes \ref{prop1:iii} true.

We finally prove \ref{prop1:iv}. For every $\omega\in \Lambda$, there exist $\xb_{\infty} (\omega)\in \Gamma$ and a subsequence $\left(\xv_{\psi(k)}(\omega)\right)_{k \in \eN}$ such that 
$\xv_{\psi(k)}(\omega)\underset{k \to + \infty}{\longrightarrow} \xb_{\infty}(\omega)$,
and $V(\xv_{\psi(k)}(\omega))\underset{k \to + \infty}{\longrightarrow} \VD_{\infty}(\omega)$. 
In addition, since $V$ is continuous, we deduce that $V(\xv_{\psi(k)}(\omega))\underset{k \to + \infty}{\longrightarrow} V\left(\xb_{\infty}(\omega)\right)$. 
Hence $\VD_{\infty}(\omega) = V\left(\xb_{\infty}(\omega)\right) \in V(\Gamma)$ which concludes the proof.
\end{proof}

\section{$\text{K\L}$ theory as a baseline of improvement}
\label{sec3}

In this section we present cornerstone results that will be used in the remainder of the paper to conduct the theoretical analysis of stochastic schemes. 
Specifically, we propose a new use of the $\text{K\L}$   property that does not explicitly depend on $\omega$ variable. 
This will allow us to derive new asymptotic results leveraging mathematical tools different from the standard ODE method and ``trajectory-by-trajectory" strategy used in \cite{benaim1996dynamical,dereich2021convergence, tadic2015convergence}.

To this aim, we first give some notation in Section~\ref{sec3.0} and recall the $\text{K\L}$   framework introduced in \cite{Bolte2014} in Section~\ref{sec3.1}. We then provide an extended version of this framework in Section~\ref{sec3.2}, and we finally introduce in Section~\ref{sec3.3} the stochastic framework that will enable us to us the $\text{K\L}$   property to study a.s. convergence of stochastic differentiable schemes under mild assumptions.

\subsection{A set of concave desingularization functions}
\label{sec3.0}

In this section we introduce a specific class of functions that will be necessary to reframe the $\text{K\L}$   property into a stochastic framework.

\begin{definition}\cite{bolte2010characterizations}\label{def:desing}
    Let $\zeta \in (0,+\infty]$. We denote $\Phi_{\zeta}$ the set of \textit{concave desingularization functions} defined as the set of concave functions $\varphi\colon [0,\zeta) \mapsto \mathbb{R}_{+}$ such that $\varphi(0) = 0$, $\varphi$ is continuous in $0$, $\varphi \in C^1((0,\zeta))$, and, for every $ s \in (0,\zeta)$, $\varphi'(s) > 0$.
\end{definition}

In Definition~\ref{def:desing}, we can have $\zeta = +\infty$. 
The link between $\Phi_{\zeta}$ for $\zeta \in [0, +\infty)$ and $\Phi_{+\infty}$ is described by the following proposition which we will use much later for essentially technical purposes in section \ref{sec4}. However, we have preferred to present the latter in this subsection for the sake of consistency.

\begin{proposition}\label{prop3}
Let $\zeta\in(0,+\infty)$. Any function $\varphi$ of $\Phi_{\zeta}$ admits a bounded extension $\widetilde{\varphi}$ belonging to $\Phi_{+\infty}$.  
\end{proposition}

\vspace{0.4cm}

\begin{proof}
Let $\varphi \in \Phi_{\zeta}$. Since $\varphi'$ is positive, it follows that $\varphi$ is an increasing function. 
Then $l_1= \lim_{s\to \zeta^{-} }\varphi(s)$ exists and lies in $[0,+\infty]$. Moreover the concavity and differentiability of $\varphi$ with $\varphi(0) = 0$ ensure that, for every $s\in (0,\zeta)$, $\varphi(s) \leq s \varphi'(0)$.
Passing to the limit 
thus gives $l_1\leq \zeta \varphi'(0) $ and then $l_1<+\infty$. 

Moreover, since $\varphi'$ is decreasing on $(0,\zeta)$ (due to the concavity of $\varphi$) and positive, we conclude that $l_2= \lim_{s\to \zeta^{-} }\varphi'(s)$ exists and lies in $[0,+\infty)$.

Finally, it remains easy to verify that function $\widetilde{\varphi} \colon [0, +\infty) \to \eR_+$ defined as
\begin{equation}\label{prop2_2}
    (\forall s\geq 0)\quad 
    \widetilde{\varphi}(s) = 
    \begin{cases}
    \varphi(s)  &\text{ if } s\in [0,\zeta),\\
    l_1+l_2 \zeta\left(1-\frac{\zeta}{s}\right) &\text{ otherwise}.
    \end{cases}
\end{equation}
$\widetilde{\varphi}$ belongs to $\Phi_{+\infty}$ and is bounded (see figure \ref{fig:KL_prolongement} below for an illustrative example).
\end{proof}

\begin{figure}[h]
    \centering
    \includegraphics[scale=0.32]{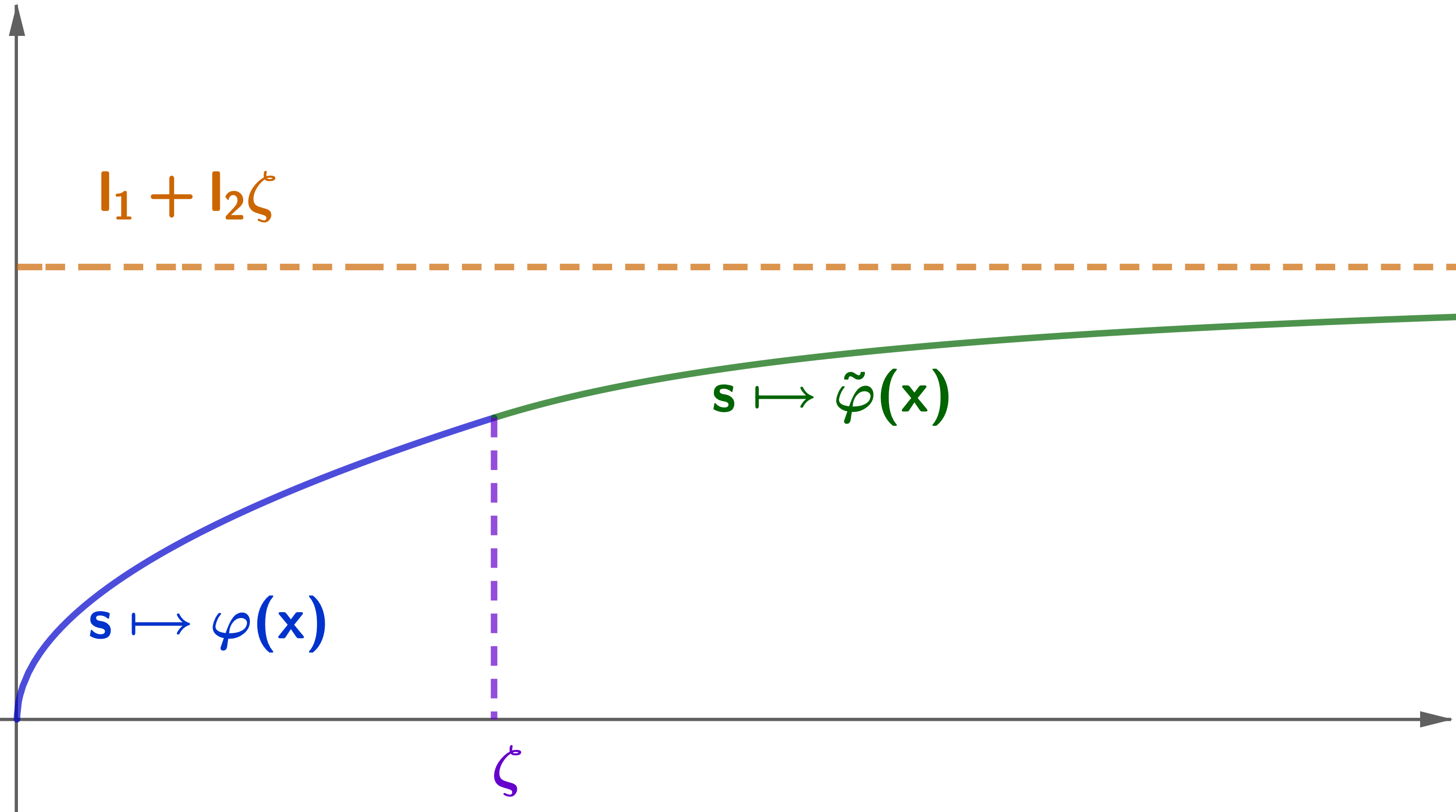}
    \caption{ Illustration of the proof of Proposition \ref{prop2_2}. $\widetilde{\varphi}$ (green curve) is the $C^1$ bounded extension over $(0,+\infty)$ of $s \in [0,\zeta) \mapsto \varphi(s) = \sqrt{s}$ (blue curve), as defined in \eqref{prop2_2}. 
    }
    \label{fig:KL_prolongement}
\end{figure}

\subsection{Limitation of the usual $\text{K\L}$   framework}
\label{sec3.1}

The $\text{K\L}$   property has been developed for both differentiable \cite{kurdyka1998gradients} and non-differentiable \cite{bolte2007clarke,attouch2010proximal} functions, to be used to study gradient descent and proximal-based algorithms, respectively, in a deterministic context (see e.g., \cite{Bolte2014,chouzenoux2014variable}).  
In this section, we focus on the differentiable $\text{K\L}$   framework, generalized to be used for stochastic schemes. Note that this focus is made mainly for the sake of simplicity, but a similar study could be done for the non-smooth framework by  replacing the gradient by the limiting subdifferential \cite{attouch2010proximal}. Section~\ref{sec5} will be dedicated to examples where the proposed theoretical differentiable framework can be used.

\begin{definition}\label{def_KL}\emph{[$\text{K\L}$   Property]}
A differentiable function $f : \Hc \to \Rbb$  satisfies the \emph{Kurdyka-Lojasiewicz ($\text{K\L}$  ) property} at $\widetilde{\xb} \in \Hc$, if there exist a neighbourhood $\Vc$ of $\widetilde{\xb}$, $\zeta > 0$ and $\varphi \in \Phi_{\zeta}$ such that 
$$\| \nabla f(\xb)\|_{\Hc}~\varphi'(f(\xb)-f(\widetilde{\xb})) \ge 1,$$ 
for every $\xb \in  \Vc $ satisfying $ 0 < f(\xb)-f(\widetilde{\xb}) < \zeta$. 
Furthermore, $f$ is said to satisfy the $\text{K\L}$   property over $E\subset \Hc$ if the latter is verified at every point of $E$. 
\end{definition}

The $\text{K\L}$   property given in Definition~\ref{def_KL} is defined locally in the sense that parameter $\zeta$ and function $\varphi$ depend on the point we are located. As a consequence, this naturally tends to favour trajectory-by-trajectory strategies as investigated, e.g., in \cite{milzarek2023convergence}. 
Instead, we aim to adopt a more global approach, as proposed in \cite[Lemma 6]{Bolte2014}, leading to the ``uniformized $\text{K\L}$   property".

\begin{theorem}\label{th_uniform_KL} \emph{[Uniformized $\text{K\L}$   property]}
Let $C$ be a compact subset of $\Hc$ and $f \colon \Hc \to \Rbb$ be a differentiable function constant on $C$, satisfying the $\text{K\L}$   property over $C$. Then, there exist $(\varepsilon, \zeta) \in (0,+\infty)^2$ and $\varphi \in \Phi_{\zeta}$ such that 
$$\label{eq:ineqKL_uniform}
     \| \nabla f(\xb)\|_{\Hc}~\varphi'(f(\xb)-f(\overline{\xb})) \geq 1,
$$
for all $\overline{\xb} \in C$ and $\xb\in \Hc$ verifying $\dist_{\Hc}(\xb,C)<\varepsilon$ and $0<f(\xb)-f(\overline{\xb}) <\zeta$.
\end{theorem}

\smallbreak

We can notice that the sets $C$ in Theorem~\ref{th_uniform_KL} and $\Gamma$ in Assumption~\ref{ass:H2} are both partially characterized by their image through $f$ and $V$, respectively.
This similarity suggests that Theorem~\ref{th_uniform_KL} can be applied to function $V$, assuming that it satisfies the $\text{K\L}$   property on $ C = \Gamma$. 
However, Theorem~\ref{th_uniform_KL} might be over restrictive as it would require that $V(\Gamma)$ is reduced to a singleton. 
Since $\Gamma$ links all trajectories $\big( \chi_\infty(\omega) \big)_{\omega\in \Omega}$, the singleton assumption would require that the image sets $\big( V(\chi_\infty(\omega)) \big)_{\omega\in \Omega}$ are themselves all reduced to a common singleton (i.e. independent from $\omega$-variable), that is not amenable to working in a non-convex context. 
Hence, the uniformized $\text{K\L}$   property given in Theorem~\ref{th_uniform_KL} can be applied to each $\chi_\infty(\omega)$, for $\omega \in \Omega$, but then the associated parameters $\zeta$ and $\epsilon$, as well as the desingularizing function $\varphi$ would all also depend on $\omega$. Such an approach would again require working using a trajectory-by-trajectory strategy. \\
As emphasized in the Introduction, in this work we aim to move away from such a strategy and develop a simplified framework to study a.s. convergence of stochastic processes.
Specifically, in the remainder of this section, we build a new version of Theorem~\ref{th_uniform_KL} such that (i) the image of $C$ by $f$ shall not be constant anymore, and (ii) $\zeta,\epsilon, \varphi$ shall remain deterministic (i.e., independent from $\omega$).
The $\omega$-dependency issue of these three last parameters has already been considered in \cite{driggs2020spring}. In their work, the authors assume a mini-batch structure of the gradient-noise and give a version in expectation of Theorem~\ref{th_uniform_KL}. 
In contrast, we aim here to provide a version of Theorem~\ref{th_uniform_KL} that is independent from the considered stochastic scheme.

\subsection{Proposed extension of the uniformized $\text{K\L}$   theorem}
\label{sec3.2}

In this section, we will give our extended version of Theorem~\ref{th_uniform_KL} to the case where the set $C$ can be taken as a finite union of compact sets on which $f$ is constant. 
This result can be seen as a combination of the uniform $\text{K\L}$   property and \cite[Corollary 11]{bolte2007clarke}, for a more generic set $C$. 
Such an extension will be necessary to provide a framework where the $\text{K\L}$   property can be used independently from the considered stochastic scheme.

\begin{theorem}\label{th_uniform_KL_bis} \emph{[Extended uniformized $\text{K\L}$   property]}
Let $C=\bigcup_{i=1}^I C_i$ be a union of $I \in \Nbb^*$ non-empty disjoint and compact subsets $(C_i)_{1 \le i \le I}$ of $\Hc$, and $f \colon \Hc \rightarrow \Rbb$ be a differentiable function satisfying the $\text{K\L}$   property on $C$. 
We also suppose that $f$ is constant on every $C_i$, for $ i \in \{1, \ldots, I\}$, with respective values $f_{C_1}, \ldots, f_{C_I}$. 
Then, there exist $(\varepsilon, \zeta) \in (0, +\infty)^2$ and $\varphi \in \Phi_{\zeta}$ such that for every $i \in \{ 1,\ldots, I \}$, and for every $\xb\in \Hc$ verifying $\dist_{\Hc}(\xb,C)<\varepsilon$ and $0<f(\xb)-f_{C_i} <\zeta$, we have
\begin{equation}\label{eq:ineqKL_uniform_ext}
    \|\nabla f(\xb)\|\varphi'(f(\xb)-f_{C_i}) \geq 1. 
\end{equation}
\end{theorem}

\vspace{0.4cm}

\begin{proof}
Without loss of generality we consider that $f_{C_1}\neq \cdots \neq f_{C_I}$. 

We start by applying the uniformized $\text{K\L}$   property \ref{th_uniform_KL} to $C_1,\ldots,C_I$. 
For every $i\in \{ 1,\ldots, I \}$, there exist $(\varepsilon_i,\zeta_i) \in (0, +\infty)^2$ and $\varphi_i\in \Phi_{\zeta_i}$ such that for every $\xb\in \Hc$ verifying $\dist_{\Hc}(\xb,C_i)<\varepsilon_i$ and $0<f(\xb)-f_{C_i}<\zeta_i $ we have 
\begin{equation}\label{AltUniKL_1}
    \|\nabla f(\xb) \| \varphi_i'(f(\xb)-f_{C_i})\geq 1.
\end{equation}
Let $\widetilde{\zeta} = \upsilon \min_{j\neq i}|f_{C_i}-f_{C_j}|$ with $\upsilon \in (0, 1/2)$. We have $\widetilde{\zeta}>0$. 
In addition, using the continuity of $f$, for every $i\in \{ 1,\ldots, I \}$, there exists $\widetilde{\varepsilon}_i\in(0, \varepsilon_i)$ such that, for every $\xb$ satisfying $\dist_{\Hc}(\xb,C_i)<\widetilde{\varepsilon}_i$, we have
\begin{equation}\label{AltUniKL_2}
|f(\xb)-f_{C_i}|< \widetilde{\zeta}.
\end{equation}
Let $\delta \in (0,1)$ and $\varepsilon = \delta \min_{1\leq i \leq I} \widetilde{\varepsilon}_i$. Then
\begin{equation}\label{AltUniKL_3}
   \left\{\xb\in \Hc \, \mid \, \dist_{\Hc}(\xb,C)< \varepsilon \right\} \subset \bigcup_{i=1}^{I} \left\{\xb\in \Hc\, \mid \, \dist_{\Hc}(\xb,C_i)< \widetilde{\varepsilon}_i \right\} .
\end{equation}

We now show that \eqref{eq:ineqKL_uniform_ext} is satisfied for $\varepsilon$ as defined above, $\zeta=\min(\zeta_1,\ldots\zeta_I,\widetilde{\zeta})$ and $\varphi=\sum_{i=1}^I \varphi_i \in \Phi_{\zeta}$.
Let $\xb\in \Hc$ and $i\in \{ 1,\ldots, I \}$ be such that $\dist_{\Hc}(\xb,C)< \varepsilon$ and $0<f(\xb)-f_{C_i}< \zeta$. 
For every $j\in \{1, \ldots, I\} \setminus \{i\}$, since $\zeta\leq \widetilde{\zeta}$, using the definition of $\widetilde{\zeta}$, we have $0<|f(\xb)-f_{C_i}|\leq v |f_{C_i}-f_{C_j}|$.
Insofar $v\in (0,1/2)$, we obtain
\begin{align}
    | f(\xb)-f_{C_j} |
    &   = | (f(\xb)-f_{C_i}) + (f_{C_i} - f_{C_j}) | \notag \\
    &   \geq  \left|~|f_{C_i}-f_{C_j}| - |f(\xb)-f_{C_i}|~\right| \notag \\
    &= |f_{C_i}-f_{C_j}| - |f(\xb)-f_{C_i}| \notag \\
    &   \geq (1-v) |f_{C_i}-f_{C_j}| \notag \\
     &   > v |f_{C_i}-f_{C_j}|\notag \\
    &   \geq \widetilde{\zeta}.
\end{align}
Then $\dist_{\Hc}(\xb,C_j) \geq \widetilde{\varepsilon_j}$, otherwise this contradicts the continuity property of $f$ in \eqref{AltUniKL_2}. 
So, as $\dist_{\Hc}(\xb,C)< \varepsilon$, and due to \eqref{AltUniKL_3}, this necessarily implies that 
$\dist_{\Hc}(\xb,C_i)< \widetilde{\varepsilon}_i$ and thus $\dist_{\Hc}(\xb,C_i)< \varepsilon_i$ (since $\widetilde{\varepsilon_i} \leq \varepsilon_i)$. 
In addition, since $\zeta \leq \zeta_i$, we have $0<f(\xb)-f_{C_i}< \zeta_i$, and we can apply once more the uniformized $\text{K\L}$   property at $C_i$ (i.e., \eqref{AltUniKL_1} is verified). Finally the positivity of $\varphi_1',\ldots,\varphi_I'$ ensure the desired relation as follows
\begin{align*}
    \|\nabla f(\xb) \| \varphi'(f(\xb)-f_{C_i})
    &=  \|\nabla f(\xb) \| \sum_{j=1}^I \varphi_j'(f(\xb)-f_{C_i}) \nonumber \\
    &\geq  \|\nabla f(\xb) \|  \varphi_i'(f(\xb)-f_{C_i}) \geq 1 .
\end{align*}
This completes the proof.
\end{proof}

\subsection{A $\text{K\L}$   framework for stochastic schemes}
\label{sec3.3}

The objective of this section is to introduce a new framework derived from our extended uniformized $\text{K\L}$   property (Theorem~\ref{th_uniform_KL_bis}), that can be used in a stochastic setting, independently from the considered stochastic scheme (in particular removing the $\omega$-dependency that appears if using directly Theorem~\ref{th_uniform_KL}). 
To this aim we need to work under the following assumption.

\begin{assumption}\label{ass:H3}\
\begin{enumerate}
    \item \label{ass:H3:i}
    The function $V$ defined in Assumption~\ref{ass:H1} satisfies the $\text{K\L}$   property (on $\Hc$).
    \item \label{ass:H3:ii}
    There exists a positive integer $I<+\infty$ such that the set $\Gamma$ defined in Assumption~\ref{ass:H2}\ref{ass:H2:ii} satisfies $\Gamma=\bigcup_{i=1}^I C_i $, where, for every $i\in \{1, \ldots, I\}$, $C_i$ is a compact set on which $V$ is constant on. 
\end{enumerate}
\end{assumption}

Note that assuming that $V$ satisfies the $\text{K\L}$   property on $\Hc$ (i.e., Assumption~\ref{ass:H3}\ref{ass:H3:i}) is common in non-convex optimization. As emphasized, e.g., in \cite{bolte2008characterizations}, the $\text{K\L}$   inequality is satisfied for a wide class of functions, and in particular by real analytic, semi-algebraic\footnote{A function is semi-algebraic if its graph is a finite union of sets defined by a finite number of polynomial inequalities.} and log-exp functions.

\begin{proposition}\label{prop_stoKL}
Assume that  Assumptions \ref{ass:H1}, \ref{ass:H2} and \ref{ass:H3} hold and assume that the event 
\begin{equation}\label{def:Pi}
\Pi_V := \liminf\limits_{k\to +\infty}\left\{\omega \in \Omega \, \mid \, V(\xv_k(\omega))>\VD_{\infty}(\omega) \right\}.
\end{equation} 
has probability one, i.e. $\Pbb(\Pi_V)=1$.
Then there exist a bounded function $\varphi \in \Phi_{+\infty}$ and an a.s. finite positive discrete random variable $\KD$ such that 
\begin{equation}
(\forall k > \KD)\quad
\|\nabla V(\xv_k)\|  \varphi'(V(\xv_k) - \VD_{\infty})\geq 1~~\mathrm{a.s.}.
\end{equation}
\end{proposition}

\vspace{0.2cm}

\begin{proof}
The cornerstone of the proof relies in the application of Theorem~\ref{th_uniform_KL_bis} on set
\begin{equation}\label{stoKL_1}
    \mathcal{C}:= \bigcup_{\omega \in \Theta\cap \Pi_V} \chi_\infty(\omega),
\end{equation}
where 
\begin{equation} \label{stoKL_1a}
 \Theta = \Big\{V(\xv_k) \underset{k\to +\infty}{\to} \VD_{\infty}\Big\} \bigcap  \Big\{\VD_{\infty}\in V(\Gamma)\Big\}
 \bigcap \Big\{\dist_{\Hc}(\xv_k,\chi_\infty)\underset{k\to +\infty}{\to} 0 \Big\}.
\end{equation}
Here, the role of $C$ is to bring together the accumulation points of all feasible trajectories of process $(\xv_k)_{k\in \Nbb}$. This will allow us to build a desingularization function which is uniform in $\omega$-variable.

First, by definition of $C$ \eqref{stoKL_1} and continuity of $V$, we have
\begin{equation*}
    V(\Cc)=V\left(\bigcup_{\omega \in \Theta\cap \Pi_V} \chi_\infty(\omega)\right)=\bigcup_{\omega \in \Theta\cap \Pi_V} V \left(\chi_\infty(\omega)\right).
\end{equation*}
Furthermore, since $V$ is continuous, by definition of $\VD_\infty$ and $\chi_\infty$ we have $V \left(\chi_\infty(\omega)\right) = \left \{ \VD_{\infty}(\omega)\right\} $, and hence
\begin{equation}\label{stoKL_1b}
    V(\Cc)= \bigcup_{\omega \in \Theta\cap \Pi_V} \left \{  \VD_{\infty}(\omega)\right\} 
    \subset \bigcup_{\omega \in \Theta\cap \Pi_V} V(\Gamma) = V(\Gamma),
\end{equation}
where the inclusion is obtained by~\eqref{stoKL_1a} and the last equality is due to the fact that $V(\Gamma)$ is constant.
Since $|V(\Gamma)|<+\infty$ (Assumption~\ref{ass:H2}\ref{ass:H2:i}), it follows from \eqref{stoKL_1b} that $|V(\Cc)|<+\infty$.
Then Proposition~\ref{prop0} ensures that $\overline{\Cc}$ can be written as a finite union of non-empty compacts on which $V$ is constant on.   

Since $V$ satisfies the $\text{K\L}$   property, we can apply Theorem \ref{th_uniform_KL_bis} considering $f=V$ and $C=\overline{\Cc}$. 
Let $I=|V(\overline{\Cc})|$ and $V(\overline{\Cc})=\{V_1, \ldots, V_I \}$ with $V_1\neq V_2 \cdots \neq V_I$. Then, there exist $\varepsilon_{\mathcal{C}}>0$, $\zeta_{\mathcal{C}}>0$ and $\varphi_{\mathcal{C}}\in \Phi_{\zeta_{\mathcal{C}}}$, such that, for every $\xb\in \Hc$ and $i\in I$ satisfying  
$\dist_{\Hc}(\xb,\overline{\Cc})<\varepsilon_{\Cc}$ and $0<V(\xb)-V_i <\zeta_{\mathcal{C}}$, 
\begin{equation}\label{stoKL_2}
    \| \nabla V(\xb)\|_{\Hc}~\varphi_{\mathcal{C}}'(V(\xb)-V_i) \geq 1.
\end{equation}
According to Proposition \ref{prop3}, $\varphi_{\mathcal{C}}$ has a bounded extension $\widetilde{\varphi}_{\mathcal{C}}$ belonging to $\Phi_{+\infty}$. Then, $\widetilde{\varphi}_{\mathcal{C}}$ also satisfies, for any $\xb\in \Hc$ and $i\in I$ such that  
$\dist_{\Hc}(\xb,\overline{\Cc})<\varepsilon_{\mathcal{C}}$ and $0<V(\xb)-V_i<\zeta_{\mathcal{C}}$, 
\begin{equation}\label{stoKL_2_bis}
     \| \nabla V(\xb)\|_{\Hc}~\widetilde{\varphi}_{\mathcal{C}}'(V(\xb)-V_i) \ge 1.
\end{equation}
We now introduce the positive discrete random variable
\begin{equation}\label{stoKL_3}
    \KD = \min \left\{l >0 \, \mid \, (\forall p\geq l) \; \dist_{\Hc}(\xv_p,\overline{\Cc})\leq \varepsilon_{\mathcal{C}} \text{ and } 0<V(\xv_p)-\VD_{\infty}<\zeta_{\mathcal{C}} \right\}.
\end{equation}
According to Assumption~\ref{ass:H2}\ref{ass:H2:i} and Proposition~\ref{prop1}, for every $\omega \in \Theta\cap \Pi_V$, $\KD(\omega)<+\infty$. 
In addition, for every $\omega \in \Theta\cap \Pi_V$, since $\VD_{\infty}(\omega)\in V(\overline{\Cc})=\{V_1, \ldots, V_I\}$, there exists a unique $i_{\omega} \in I$ such that $\VD_\infty(\omega) = V_{i_{\omega}}$. As a consequence, \eqref{stoKL_2_bis} finally leads to  
\begin{equation}
     (\forall \omega \in \Theta\cap \Pi_V)\quad (\forall k>\KD(w))\quad  \| \nabla V(\xb(\omega))\|_{\Hc}~\widetilde{\varphi}_{\mathcal{C}}'\left(V(\xv_k(\omega))-\VD_{\infty}(\omega)\right) \geq 1.
\end{equation}
Assumption \ref{ass:H2} and Proposition \ref{prop1} lead to $\Pbb(\Theta)=1$. Furthermore by assumption we have $\Pbb(\Pi_V)=1$. Combining the two leads to $\Pbb(\Theta\cap \Pi_V)=1$ and so concludes the proof.   
\end{proof}

\section{An almost sure convergence result based on $\text{K\L}$ theory}
\label{sec4}
This section is devoted to the implementation of our theoretical approach based on conditioning theory (similarly to~\cite{Bolte2014} in the deterministic case), combined with Proposition~\ref{prop_stoKL} from the previous section. 


In the following, we consider a slightly more restrictive assumption than Assumption~\ref{ass:H2}\ref{ass:H2:i}, given below.
\begin{assumption}\label{ass:H4}
Process $(V(\xv_k))_{k\in \Nbb}$ is a $\Fc_k$-supermartingale on some probability space $\left(\Omega, \Fc, \Pbb \right)$, given a filtration $(\Fc_k)_{k\in \Nbb}$.  
\end{assumption}
Note that if Assumptions~\ref{ass:H1} and \ref{ass:H4} are satisfied, then Assumption~\ref{ass:H2}\ref{ass:H2:i} holds.

In this section, we first give a technical result in Section~\ref{Ssec:sec4:1}, that will be used to apply the $\text{K\L}$   property in our main result given in Section~\ref{Ssec:sec4:2}.

\subsection{A useful technical lemma}\label{Ssec:sec4:1}
 
We introduce here a purely technical lemma which will be useful for establishing the central result of this section 

\begin{lemma}\label{lemma:sum}
Under Assumption \ref{ass:H1} and \ref{ass:H4},
\begin{enumerate}
\item \label{lemma:sum:1} Process $(V(\xv_k))_{k\in \Nbb}$ converges almost-surely to an integrable random variable $\VD_{\infty}$. 
\item \label{lemma:sum:2} Assume that, for all $k\in \Nbb$, either
\begin{equation}\label{lemma:sum:2:cond1}
    V(\xv_k)\geq \Ebb[\VD_{\infty}|\Fc_k]>\VD_{\infty} \text{ a.s.},
\end{equation}
or
\begin{equation}\label{lemma:sum:2:cond2}
    V(\xv_k)> \Ebb[\VD_{\infty}|\Fc_k]\geq \VD_{\infty} \text{ a.s.}
\end{equation}
Then for any bounded function $\varphi \in \Phi_{+\infty}$ 
\begin{align}
    & \Ebb[\XD_{\varphi}] \leq \Ebb[\varphi(V(\xv_0) - \VD_{\infty})] \label{eq:lemma:sum:1}, \\ & \text{where}\quad \XD_{\varphi} := \sum_{k=0}^{+\infty} \varphi'\left(V(\xv_k) - \VD_{\infty}\right) \left(V(\xv_k)- \Ebb[V(\xv_{k+1})|\Fc_k] \right)\label{eq:lemma:sum:2}.
\end{align}
\end{enumerate}
\end{lemma}

\begin{proof}
\begin{enumerate}
\item 
Since $V$ is lower-bounded as a coercive and continuous function (Assumption~\ref{ass:H1}), \ref{lemma:sum:1} holds as a direct consequence of the Doob's convergence theorem \cite{williams1991probability}. 
\smallbreak

\item 
Let $k\in \Nbb$.
By hypothesis, we have $V(\xv_k) - \VD_{\infty}>0$, $\Ebb[V(\xv_{k+1})|\Fc_k]) - \VD_{\infty} \ge 0$ and $V(\xv_k)-\Ebb[V(\xv_{k+1})|\Fc_k]) \ge 0$. Hence the three quantities $\varphi(V(\xv_k) - \VD_{\infty})$, $\varphi(\Ebb[V(\xv_{k+1})|\Fc_k]) - \VD_{\infty})$ and $\varphi'(V(\xv_k) - \VD_{\infty})$ are well defined. 
Further, the concavity of $\varphi$ leads to: 
\begin{multline}\label{proof:lemma:sum:1}
 \varphi(V(\xv_k) - \VD_{\infty}) - \varphi(\Ebb[V(\xv_{k+1})|\Fc_k]) - \VD_{\infty})   \\
    \geq \varphi'(V(\xv_k) - \VD_{\infty}) \left(V(\xv_k)-\Ebb[V(\xv_{k+1})|\Fc_k] \right)~~\mathrm{a.s.}. 
\end{multline}
Using the fact that $\VD_{\infty}\leq \Ebb[\VD_{\infty}|\Fc_k]$ and $\VD_{\infty}\leq V(\xv_{k+1})$ almost-surely, we can deduce that 
\begin{align}\label{proof:lemma:sum:2}
    &\Ebb[V(\xv_{k+1})|\Fc_k]-\VD_{\infty} \notag \\
    &\geq \Ebb[V(\xv_{k+1})|\Fc_k]-\Ebb[\VD_{\infty}|\Fc_k]= \Ebb[V(\xv_{k+1})-\VD_{\infty}|\Fc_k]\geq 0~~\mathrm{a.s.}.
\end{align}
Then, we can apply $\varphi$ to \eqref{proof:lemma:sum:2} (as an increasing function), and we obtain
\begin{equation*}
    \varphi\left(\Ebb[V(\xv_{k+1})|\Fc_k]-\VD_{\infty}\right)\geq \varphi \left(\Ebb[V(\xv_{k+1})-\VD_{\infty}|\Fc_k] \right)~~\mathrm{a.s.}.
\end{equation*}
Moreover, due to Jensen's inequality:
\begin{equation*}
    \varphi \left(\Ebb[V(\xv_{k+1})-\VD_{\infty}|\Fc_k] \right) \geq \Ebb\left[\varphi\left(V(\xv_{k+1})-\VD_{\infty}\right)|\Fc_k\right]~~\mathrm{a.s.}.
\end{equation*}
Hence combining the last two inequalities we get
\begin{equation}\label{proof:lemma:sum:3bisbis}
    \varphi\left(\Ebb[V(\xv_{k+1})|\Fc_k]-\VD_{\infty}\right)\geq \Ebb\left[\varphi\left(V(\xv_{k+1})-\VD_{\infty}\right)|\Fc_k\right]~~\mathrm{a.s.}.
\end{equation}
Injecting \eqref{proof:lemma:sum:3bisbis} in \eqref{proof:lemma:sum:1} gives:
\begin{multline}\label{proof:lemma:sum:4}
\varphi(V(\xv_k) - \VD_{\infty}) - \Ebb\left[\varphi\left(V(\xv_{k+1})-\VD_{\infty}\right)|\Fc_k\right]  \\
    \geq \varphi'(V(\xv_k) - \VD_{\infty}) \left(V(\xv_k)-\Ebb[V(\xv_{k+1})|\Fc_k] \right)~~\mathrm{a.s.}. 
\end{multline}
The two terms involved in \eqref{proof:lemma:sum:4} being non-negative, we are allowed to pass to expectation operator. Then, by linearity of the expectation (as $\varphi$ is bounded), we obtain 
\begin{multline}\label{proof:lemma:sum:5}
  \Ebb\left[\varphi(V(\xv_k) - \VD_{\infty})\right] - \Ebb\left[\varphi(V(\xv_{k+1}) - \VD_{\infty})\right]  \\
     \geq \Ebb\left[\varphi'(V(\xv_k) - \VD_{\infty}) \left(V(\xv_k)-\Ebb[V(\xv_{k+1})|\Fc_k] \right)\right]~~\mathrm{a.s.}. 
\end{multline}
On the one hand, by summing \eqref{proof:lemma:sum:5} for $ k \in \Nbb$, we have
\begin{align}\label{proof:lemma:sum:6}
  &\sum_{k=0}^{+\infty}\left(\Ebb\left[\varphi(V(\xv_k) - \VD_{\infty})\right] - \Ebb\left[\varphi(V(\xv_{k+1}) - \VD_{\infty})\right]\right) \notag \\
    &\qquad\quad 
        \geq \sum_{k=0}^{+\infty}\Ebb\left[\varphi'(V(\xv_k) - \VD_{\infty}) \left(V(\xv_k)-\Ebb[V(\xv_{k+1})|\Fc_k] \right)\right] \notag \\
    &\qquad\quad 
        =\Ebb\left[\sum_{k=0}^{+\infty}\varphi'(V(\xv_k) - \VD_{\infty}) \left(V(\xv_k)-\Ebb[V(\xv_{k+1})|\Fc_k] \right)\right],
\end{align}
where the last equality holds since $\varphi'\geq 0$.
\newline
On the other hand, since $V(\xv_k)\underset{k\to +\infty}{\to}{\VD_{\infty}}$ a.s. we have
\begin{equation*}
  \Ebb\left[\varphi(V(\xv_k) - \VD_{\infty})\right]\underset{k\to +\infty}{\to}\Ebb\left[\varphi(\VD_{\infty} - \VD_{\infty})\right]=0,
\end{equation*}
by virtue of the dominated convergence theorem (as $\varphi$ is bounded). 
Hence
\begin{equation}\label{proof:lemma:sum:7}
\sum_{k=0}^{+\infty}\left(\Ebb\left[\varphi(V(\xv_k) - \VD_{\infty})\right] - \Ebb\left[\varphi(V(\xv_{k+1}) - \VD_{\infty})\right]\right)
    =\Ebb\left[\varphi(V(\xv_0) - \VD_{\infty})\right].
\end{equation}
As a consequence, combining \eqref{proof:lemma:sum:7} and \eqref{proof:lemma:sum:6} gives 
\begin{align}
  &\Ebb\left[\varphi(V(\xv_0) - \VD_{\infty})\right] \notag \\
  & \geq \Ebb\left[\sum_{k=0}^{+\infty}\varphi'(V(\xv_k) - \VD_{\infty}) \left(V(\xv_k)-\Ebb[V(\xv_{k+1})|\Fc_k] \right)\right] = \Ebb[\XD_{\varphi}],
\end{align}
which concludes the proof.
\end{enumerate}
\end{proof}

Conditions \eqref{lemma:sum:2:cond1} and \eqref{lemma:sum:2:cond2} in Lemma~\ref{lemma:sum} may look technical. 
On the one hand, the left hand-side inequalities (of the form $ V(\xv_k)\geq \Ebb[\VD_{\infty}|\Fc_k]$) always hold in the context of supermartingales, as a direct consequence of Fatou's lemma. 
However, on the other hand, the right hand-side inequalities (of the form $\Ebb[\VD_{\infty}|\Fc_k]\geq \VD_{\infty}$) are more technical, and indicates that the estimation of $\VD_{\infty}$ with respect to $\Fc_k$, i.e. $\Ebb[\VD_{\infty}|\Fc_k]$, must remain greater than the true limit value $\VD_{\infty}$ (in other word, we must ``approach" this limit ``from above").

\subsection{Main result}\label{Ssec:sec4:2}

We are now ready for our main convergence theorem, that aims to provide a generic framework to show convergence of stochastic schemes in a differentiable non-convex context. To this aim, we need to combine Lemma~\ref{lemma:sum} with Proposition~\ref{prop_stoKL}.

\begin{theorem}\label{thKL}
Assume that Assumptions \ref{ass:H1}, \ref{ass:H2}\ref{ass:H2:ii}, \ref{ass:H3} and \ref{ass:H4} are verified for a probability space $(\Omega, \Fc, \Pbb)$ and a filtration $(\Fc_k)_{k\in \Nbb}$. 
Assume that, for every $k\in \Nbb$, $\| \nabla V(\xv_k)\|_{\Hc}>0$, and that either \eqref{lemma:sum:2:cond1} or \eqref{lemma:sum:2:cond2} holds, for $\VD_{\infty}$ the a.s. limit defined in Lemma~\ref{lemma:sum}\ref{lemma:sum:1}.
Then the following hold.
\begin{enumerate}
\item 
There exist $\varphi\in \Phi_{+\infty}$ and an almost-sure finite discrete positive random variable $\KD$ such that the following inequality holds:
\begin{equation}\label{thKL:inequality}
    \Ebb\left[\DD\right ]\leq \Ebb\left[\varphi(V(\xv_0) - \VD_{\infty})\right] \text{ where }
    \DD:=\sum_{k=\KD+1}^{+\infty} \frac{V(\xv_k)-\Ebb[V(\xv_{k+1})|\Fc_k]}{\| \nabla V(\xv_k)\|_{\Hc}}.
\end{equation}

\item We have
\begin{equation}\label{cor:thKL:eq}
    \sum_{k=0}^{+\infty} \frac{V(\xv_k)-\Ebb[V(\xv_{k+1})|\Fc_k]}{\| \nabla V(\xv_k)\|_{\Hc}}<+\infty~~\mathrm{a.s.}.
\end{equation}
\end{enumerate}
\end{theorem}

\vspace{0.2cm}

\begin{proof}
\begin{enumerate}
\item 
Let $\Pi_V$ defined as in \eqref{def:Pi}. By hypothesis (\eqref{lemma:sum:2:cond1} or \eqref{lemma:sum:2:cond2}), this is 
a probability-one set. Thus Proposition \ref{prop_stoKL} ensures the existence of $\varphi\in \Phi_{+\infty}$ and an almost-sure finite random variable $\KD$ such that
\begin{equation}\label{thKL:pr:eq-KL}
    (\forall k\geq \KD)\quad \| \nabla V(\xv_k)\|_{\Hc}~ \varphi'\left(V(\xv_k)-\VD_{\infty}\right)\geq 1~~\mathrm{a.s.}.
\end{equation}
Let $\XD_{\varphi}$ be the variable defined in \eqref{eq:lemma:sum:2}. Using~\eqref{thKL:pr:eq-KL}, we have
\begin{align*}
    \XD_{\varphi} &=\sum_{k=0}^{\KD} \varphi'\left(V(\xv_k) - \VD_{\infty}\right) \left(V(\xv_k)- \Ebb[V(\xv_{k+1})|\Fc_k] \right) \notag \\ 
    &\qquad\quad
        + \sum_{k=\KD+1}^{+\infty} \varphi'\left(V(\xv_k) - \VD_{\infty}\right) \left(V(\xv_k)- \Ebb[V(\xv_{k+1})|\Fc_k] \right)\notag \\
    &\geq\sum_{k=0}^K \varphi'\left(V(\xv_k) - \VD_{\infty}\right) \left(V(\xv_k)- \Ebb[V(\xv_{k+1})|\Fc_k] \right) \notag \\
    & \qquad\quad  + \sum_{k=\KD+1}^{+\infty}\frac{V(\xv_k)-\Ebb[V(\xv_{k+1})|\Fc_k]}{\| \nabla V(\xv_k)\|_{\Hc}} \notag \\
    &\geq \sum_{k=\KD+1}^{+\infty}\frac{V(\xv_k)-\Ebb[V(\xv_{k+1})|\Fc_k]}{\| \nabla V(\xv_k)\|_{\Hc}}~~\mathrm{a.s.},
\end{align*}
where the last inequality is obtained using the fact that $\varphi'$ is positive and Assumption~\ref{ass:H4}.
Passing to the expectation (all quantities involved are positives almost-surely) and using Lemma~\ref{lemma:sum} leads to
\begin{equation}\label{thKL:pr:eq-sum}
    \Ebb\left[\sum_{k=\KD+1}^{+\infty}\frac{V(\xv_k)-\Ebb[V(\xv_{k+1})|\Fc_k]}{\| \nabla V(\xv_k)\|_{\Hc}}\right] \leq \Ebb[\XD_{\varphi}] \leq \Ebb\left[\varphi(V(\xv_0) - \VD_{\infty})\right].
\end{equation}

\item 

Because $\Ebb[\XD_{\varphi}]$ is finite (according to \eqref{thKL:pr:eq-sum}) we have
$$\sum\limits_{k=\KD+1}^{+\infty} \frac{V(\xv_k)-\Ebb[V(\xv_{k+1})|\Fc_k]}{\| \nabla V(\xv_k)\|_{\Hc}}<+\infty~~\mathrm{a.s.}.$$ 
On event $\{\KD<+\infty\}$, of probability one, we thus have
\begin{align*}
    &\sum_{k=0}^{+\infty} \frac{V(\xv_k)-\Ebb[V(\xv_{k+1})|\Fc_k]}{\| \nabla V(\xv_k)\|_{\Hc}} \notag \\
    &= \sum_{k=0}^{\KD} \frac{V(\xv_k)-\Ebb[V(\xv_{k+1})|\Fc_k]}{\| \nabla V(\xv_k)\|_{\Hc}} \notag \\
    &\qquad\quad + \sum_{k=\KD+1}^{+\infty} \frac{V(\xv_k)-\Ebb[V(\xv_{k+1})|\Fc_k]}{\| \nabla V(\xv_k)\|_{\Hc}}
    <+\infty~~\mathrm{a.s.}. 
\end{align*}
This concludes the proof. 
\end{enumerate}
\end{proof}

Theorem \ref{thKL} ensures the existence of an upper-bound for the expectation of the variable $\DD$ (see \eqref{thKL:inequality}) which can interpreted as follows.
The residual of the descent condition (i.e. $V(\xv_k)-\Ebb[V(\xv_{k+1})|\Fc_k]$) is expected to be ``very small" compared to $\|V(\xv_k)\|_{\Hc}$ for $k > \KD$. 
When comparing with previous works in the deterministic setting \cite{attouch2010proximal,chouzenoux2014variable,chouzenoux2023convergence}, it is consistent to have an upper-bound depending on the desingularization function $\varphi$. In the stochastic case, variable $\KD$ plays a role similar to a stopping time \footnote{Note that $\KD$ is not a stopping time in the purely mathematical sense since, as seen in the proof of Proposition~\ref{prop_stoKL}, it involves the variable $\VD_{\infty}$ which is not $\Fc_k$-measurable.} that appears to be consistent with \cite[Theorem 2.2]{tadic2015convergence}.

\section{Building  $\Hc$ space and $V$ function for generic SGD. }
\label{sec5}

In the remainder of this work, we will show how the results presented in the previous sections (in particular Theorem~\ref{thKL}) can be used to establish convergence guarantees of stochastic approximation schemes for solving problem~\eqref{eq:pb}. 

For the sake of simplicity, we will focus on the asymptotic convergence of a generic stochastic gradient descent (SGD) scheme, with preconditioning. Specifically, we define
\begin{equation}\label{eq:scheme:grad_sto}
    (\forall k \in \Nbb)\quad \wv_{k+1} = \wv_k - \alpha_k \Uv_k \fv_k
\end{equation}
where 
$(\alpha_k)_{k\in \Nbb}$ is a sequence of positive step-sizes, 
$(\fv_k)_{k\in \Nbb}$ is a process of directions belonging to $\eR^N$ aiming at approximating the gradients $\left( \nabla F(\wv_k)\right)_{k\in \Nbb}$,
and $(\Uv_k)_{k\in \Nbb}$ is a process of preconditioning matrices in $\Rbb^{N \times N}$. 
Scheme~\eqref{eq:scheme:grad_sto} encompasses the most usual second order methods of SGD \cite{bottou2018optimization} and notably, for the non-convex setting, some Newton/Quasi-Newton versions as those of \cite{mokhtari2014res,wang2017stochastic}. 
Then, our analysis will rely on the finding of a finite-dimensional Euclidean space $\Hc$, a process $(\xv_k)_{k\in \Nbb}$ and a Lyapunov-type function $V \colon \Hc \to \Rbb$ associated with the problem of interest \eqref{eq:pb}, that will enable stronger convergence guarantees on $(\wv_k)_{k\in \Nbb}$ than classical non-convex studies.

\subsection{Assumptions}
\label{subsec:setting_assumpts}

We consider a probabilistic space $(\Omega, \Fc, \Pbb)$ with the canonical filtration $\Fc_0=\{ \emptyset, \Omega\}$ and for all $k\geq 1~\Fc_k=\sigma(\fv_0,\wv_1,...,\fv_{k-1},\wv_k)$ the smallest $\sigma$-algebra gathering all the past information from zero to the current state $k$. We consider the following assumptions:

\begin{assumption}~
\label{ass:H:5}
\begin{enumerate}
\item \label{ass:H:5:i} 
$F$ is coercive and $\beta$-Lipschitz differentiable i.e. for every $(\wb,\wb') \in \eR^N \times \eR^N$, $\| \nabla F(\wb)- \nabla F(\wb')\| \leq \beta \|\wb- \wb' \|$.
\item \label{ass:H:5:ibis} 
The image of $\zer \nabla F$ by $F$ satisfies $|F(\zer \nabla F)|<+\infty$.
\item \label{ass:H:5:vi}
$\left(F(\wv_k)\right)_{k\in \Nbb}$ and $\left(\nabla F(\wv_k)\right)_{k\in \Nbb}$ are integrable processes.
\item \label{ass:H:5:ii} 
$(\Uv_k)_{k\in \Nbb}$ is $\Fc_k$-measurable, made of symmetric matrices and there exist $\underline{\nu},\overline{\nu}>0$
 such that:
\begin{equation*}
 (\forall k\in \Nbb)\quad \underline{\nu}\|\fv_k\|^2\leq \langle \fv_k, \Uv_k \fv_k\rangle \leq  \overline{\nu}\|\fv_k\|^2~~\mathrm{a.s..}
\end{equation*}
\item \label{ass:H:5:iii} 
There exist $\mu, B>0$ and $A,C\geq 0$ such that almost-surely and for all $k\in \Nbb$ 
\begin{equation*}
\begin{cases}
\displaystyle \big\langle \nabla F(\wv_k), \Uv_k \Ebb\big[ \fv_k|\Fc_k \big] \big\rangle \geq \mu\big\langle \nabla F(\wv_k), \Uv_k \nabla F(\wv_k) \big\rangle \\[0.1cm] 
\displaystyle \Ebb\big[\|\fv_k\|^2|\Fc_k\big]\leq A \big(F(\wv_k)-F^* \big) + B\| \nabla F(\wv_k)\|^2 + C,
\end{cases}
\end{equation*}
 where $F^*$ denotes the minimal value of $F$.
\item \label{ass:H:5:iv} 
The sequence $(\alpha_k)_{k\in \Nbb}$ is positive and verifies the two Robbins-Monro conditions \cite{robbins1951stochastic} $\sum\limits_{k=0}^{+\infty} \alpha_k =+\infty$ and $\sum\limits_{k=0}^{+\infty} \alpha_k^2<+\infty$.

\end{enumerate}
\end{assumption}

\begin{table}[ht]
    \centering
    \begin{tabular}{ccc||c|c|c|c|c|c|c}
    \hline
     \multicolumn{3}{c||}{Reference} & $\mu$ & $A$ & $B$ & $C$   & $\Uv_k$ & $\underline{\nu}$ & $\overline{\nu}$  \\  
     \hline\hline
     Bertsekas and Tsitsiklis & (2000) & \cite{bertsekas2000gradient} & \ding{55} & \ding{55} & \ding{51}& \ding{51}& \ding{55} & \ding{55}& \ding{55} \\
     \hline
     Schmidt and Leroux & (2013) & \cite{schmidt2013fast} & \ding{55} & \ding{55}&  \ding{51} & \ding{55} & \ding{55} &\ding{55} & \ding{55} \\ \hline
     Berahas \textit{et al} & (2016) & \cite{berahas2016multi} & \ding{55}& \ding{55}& \ding{51}& \ding{51}& \ding{51}& \ding{51}& \ding{51} \\
     \hline
     Bollapragada \textit{et al} & (2016) & \cite{bollapragada2019exact} & \ding{55}& \ding{55}& \ding{51}& \ding{51}& \ding{51}& \ding{51}& \ding{51} \\
     \hline
     Wang \textit{et al} &(2017)& \cite{wang2017stochastic}  &\ding{55}& \ding{55}&  \ding{51} & \ding{51} & \ding{51} & \ding{51} & \ding{51} \\ \hline
     Ghadimi and Lan & (2018) & \cite{ghadimi2013stochastic} & \ding{55}& \ding{55}& \ding{55}& \ding{51}& \ding{55}&
     \ding{55}& \ding{55} \\ \hline
     Bottou \textit{et al} & (2018) & \cite{bottou2018optimization} & \ding{51}& \ding{55}& \ding{51}& \ding{51}& \ding{55}&\ding{55}& \ding{55} \\
     \hline
     Gower \textit{et al} & (2019) & \cite{gower2019sgd} & \ding{55}& \ding{51}& \ding{55}& \ding{55}& \ding{55}& \ding{55}& \ding{55} \\ 
    \hline
     Jahani \textit{et al} & (2021) & \cite{jahani2021doubly} & \ding{55}& \ding{55}& \ding{51}& \ding{51}& \ding{51}& \ding{51}& \ding{51} \\
     \hline
     Chouzenoux and Fest & (2022) & \cite{chouzenoux2022sabrina} & \ding{55} & \ding{55} & \ding{51} & \ding{51} & \ding{51} & \ding{51} & \ding{51} \\
     \hline
    \end{tabular}
    \caption{
    Examples of preconditioned SGD schemes from the literature (introduced by chronological order) satisfying Assumption~\ref{ass:H:5}\ref{ass:H:5:iii}-\ref{ass:H:5:iv} and dealing with non-convex problems. 
    \newline
    Symbols in each box should be interpreted as follows; 
    \newline 
    \ding{51}: $A, C \neq 0,~~\mu, B, \underline{\nu}, \overline{\nu} \neq 1,~~\Uv_k \neq  \mathrm{I}_N.$ 
    \newline
    \ding{55}: $A, C = 0,~~\mu, B, \underline{\nu}, \overline{\nu} = 1,~~\Uv_k = \mathrm{I}_N$ (i.e. the counterpart conditions of \ding{51}).
    }
  \label{tab1}
\end{table}

Table~\ref{tab1} gives examples of SGD schemes from the literature of the form of~\eqref{eq:scheme:grad_sto} satisfying Assumption~\ref{ass:H:5}. 
In addition, a few comments can be made on these assumptions, given in the following remark. 
\begin{remark}\
\begin{enumerate}
\item 
Assumption~\ref{ass:H:5}\ref{ass:H:5:i} and \ref{ass:H:5:ibis} are directly related to the curvature of $F$, i.e. the function to minimize in \eqref{eq:pb}. The first one is common and tends to promote the existence of a descent condition for any gradient-type schemes (both in deterministic and stochastic) via the usual descent lemma \cite{bertsekas1997nonlinear}. The second one is necessary to easily verify Assumption~\ref{ass:H2}. 
\item 
Assumption \ref{ass:H:5}\ref{ass:H:5:vi} is very common, although implicit most of the time. It will ensure the integrability of some of almost-sure limits we will need to manipulate. 
\item 
Assumption \ref{ass:H:5}\ref{ass:H:5:ii} ensures that $(\Uv_k)_{k\in \Nbb}$ is \textit{well-conditioned enough} and is not over restrictive. 
First, since, for every $k\in \Nbb$, $\Uv_k$ is generally constructed using the past information, $\Fc_k$-measurability condition is not too constraining. 
Second, the uniformly bounded spectrum condition of sequence $(\fv_k)_{k\in \Nbb}$, which is essential to ensure the existence of a decreasing Lyapunov function associated with the problem, is directly satisfied in the convex case when $(\Uv_k)_{k\in \Nbb}$ approximates the second order information of $F$ (if twice-differentiable) and includes regularity terms in its structure \cite{byrd2016stochastic, yousefian2016stochastic}. In the non-convex case, it generally requires specific update rules \cite{wang2017stochastic,chen2019stochastic}. Such rules have been investigated extensively also in the deterministic case, in particular within majorization-minimization theory \cite{geman1995nonlinear,chouzenoux2010majorize}.  
\item 
The first condition in Assumption~\ref{ass:H:5}\ref{ass:H:5:iii} classically stipulates that, with respect to the past events, the preconditioned approximate gradient and the true one are oriented in relatively closed directions. This is typically verified in the situation where process $(\fv_k)_{k\in \Nbb}$ is unbiased regarding filtration $(\Fc_k)_{k\in \Nbb}$ i.e. $\Ebb[\fv_k~|~\Fc_k] = \nabla F(\wv_k)~\mathrm{a.s.}$ for all $k\in \Nbb$. 

The second condition, often called the \textit{ABC Condition} in the literature \cite{khaled2020better}, is linked to the conditional variance and indicates that it can be controlled by three constants: $A$ controlling the distance between the evaluation of $F$ and $F^*$, $B$ controlling the multiplicative tolerable error with respect to the true gradient, and $C$ controlling some additive error. 
This condition has been first introduced in \cite{bertsekas2000gradient} to investigate SGD convergence in the non-convex setting, with $A=0$. 
It has then been used in \cite{ghadimi2013stochastic} with $A=0$ and $B=1$, and in \cite{schmidt2013fast} with $A=C=0$. The latter case is often called \textit{Strong Growth Condition}.
Although this simple case allows to study SGD taking a permissive stepsize (typically a constant one), this situation is only encountered in practice when $F$ (written in the form of an empirical risk) has all its stationary points following an interpolation condition \cite{schmidt2013fast}. 
Examples of these $A,B,C$ variables encountered in the literature are summarized in Table~\ref{tab1}.

\item 
Finally, Assumption~\ref{ass:H:5}\ref{ass:H:5:iv} is the standard Robbins-Monro condition \cite{robbins1951stochastic} used to control the relative fluctuations of the variance terms, in particular generated by the additive term $C$. 
\end{enumerate}
\end{remark}

\subsection{Quasi super-martingale condition}

Similarly to the historical approach of \cite{robbins1971convergence}, the first step to show the convergence of $(\wv_k)_{k\in \Nbb}$ as defined in \eqref{eq:scheme:grad_sto} is to establish a quasi-supermartingale descent inequality as follows.

\begin{lemma}\label{lemma:grad}
    Under Assumption \ref{ass:H:5} \ref{ass:H:5:i},\ref{ass:H:5:ii}-\ref{ass:H:5:iv}, process $(\wv_k)_{k\in \Nbb}$ defined in~\eqref{eq:scheme:grad_sto} satisfies for all $k \in \Nbb$: 
\begin{align}\label{eq:descent_inequality_grad}
    &\Ebb[F(\wv_{k+1})-F^*|\Fc_k] \notag \\  
    &\leq \left(1+ A \lambda \alpha_k^2 \right)\left(F(\wv_k)-F^*\right) 
        - \alpha_k  \left(\mu \underline{\nu} - B \lambda \alpha_k \right)\|\nabla F(\wv_k)\|^2 + C \lambda \alpha_k^2 ~\mathrm{a.s.,}
\end{align}
    where $\lambda = (\beta \overline{\nu}^2)/2$.
\end{lemma}

\vspace{0.4cm}

\begin{proof}
The proof strategy for this result is very common 
to study stochastic optimization schemes in a differentiable setting. In our case, this can be seen as an easy generalization of the proof of \cite{ghadimi2013stochastic} or \cite{bottou2010large}, that can typically be found in \cite{wang2017stochastic}. 
\end{proof}

A common approach to deduce asymptotic convergence results of SGD (in a non-convex setting) using \eqref{eq:descent_inequality_grad} then consists in using the Robbins-Siegmund lemma \cite{robbins1971convergence}.
In a nutshell, this lemma almost surely ensures that $(F(\wv_{k}))_{k\in \Nbb}$ converges to a finite limit, and that $\sum_{k=0}^{+\infty}\alpha_k \| \nabla F(\wv_k)\|^2 < +\infty$ (assuming that $(\alpha_k)_{k\in \Nbb}$ is small enough). 
In this article, the framework we have developed aims to drop the square for the summability, i.e. we aim to obtain $\sum_{k=0}^{+\infty}\alpha_k \| \nabla F(\wv_k)\|<+\infty$. For this we need to build a suitable Lyapunov function $V$ to rely on  Theorem~\ref{thKL:inequality}.
Before this, we need to reformulate the algorithm's descent condition \eqref{eq:descent_inequality_grad}, obtained in Lemma~\ref{lemma:grad}. 
\begin{corollary}\label{cor:reform_descent}
    Under Assumption \ref{ass:H:5} \ref{ass:H:5:i},\ref{ass:H:5:ii}-\ref{ass:H:5:iv}, the process $(\wv_k)_{k\in \Nbb}$ defined in~\eqref{eq:scheme:grad_sto} satisfies 
\begin{multline}\label{eq:rewrite_sto_grad}
        (\forall k \in \Nbb)\quad
        \Ebb\Big[p_{k+1}^{-1}\big(F(\wv_{k+1})-F^*\big) + C\lambda r_{k+1} \big| \Fc_k\Big] \\
        \leq ~p_k^{-1}\big(F(\wv_k)-F^*\big) + C\lambda r_k  
        - p_{k+1}^{-1}\alpha_k  \big(\mu \underline{\nu} - \lambda B \alpha_k \big)\|\nabla F(\wv_k)\|^2
        ~\mathrm{a.s.}, 
    \end{multline}
    where
    \begin{equation}\label{eq:def:rk-pk}
    (\forall k \in \Nbb)\quad 
        p_k:= \prod_{i=0}^k(1+A \lambda \alpha_i^2) 
        \quad\text{ and }\quad
        r_k:= \sum_{i=k}^{+\infty} p_i^{-1}\alpha_i^2.
    \end{equation}
\end{corollary}

\vspace{0.2cm}

\begin{proof}
    We only need to show that \eqref{eq:descent_inequality_grad} implies \eqref{eq:rewrite_sto_grad}. \\
    Since $\sum\limits_{k=0}^{+\infty} \alpha_k^2 <+\infty $ and that, for every $k\in \Nbb$, $p_k \ge 1$, thus we have $\displaystyle \sum_{k=0}^{+\infty} p_k^{-1}\alpha_k^2 < +\infty$. This allows us to write a ``difference trick" on $(p_k^{-1}\alpha_k^2)_{k\in \Nbb}$ of the form of:
\begin{equation}\label{eq:diff_trick}
(\forall k \in \Nbb)~~p_k^{-1}\alpha_k^2 = \sum_{i=k}^{+\infty} p_i^{-1}\alpha_i^2 - \sum_{i=k+1}^{+\infty} p_i^{-1} \alpha_i^2 = r_k - r_{k+1}.
\end{equation}
Multiplying, for every $k\in \Nbb$, \eqref{eq:descent_inequality_grad} by $p_{k+1}^{-1}$, and using the fact that the associated sequence $(p_k^{-1})_{k\in \Nbb}$ is non-increasing, we obtain 
\begin{align*}
    &\Ebb\left[p_{k+1}^{-1}\left(F(\wv_{k+1})-F^*\right)|\Fc_k\right] \\  
    &\leq p_{k+1}^{-1}\left(1+ A \lambda \alpha_k^2 \right)\left(F(\wv_k)-F^*\right) 
    - p_{k+1}^{-1}\alpha_k  \left(\mu \underline{\nu} - B \lambda \alpha_k \right)\|\nabla F(\wv_k)\|^2 + C \lambda p_{k+1}^{-1}\alpha_k^2 \\
    &\leq p_{k+1}^{-1}\left(1+ A \lambda \alpha_k^2 \right)\left(F(\wv_k)-F^*\right) 
    - p_{k+1}^{-1}\alpha_k  \left(\mu \underline{\nu} - B \lambda \alpha_k \right)\|\nabla F(\wv_k)\|^2 + C \lambda p_{k}^{-1}\alpha_k^2.
\end{align*}
Combining this inequality with \eqref{eq:diff_trick} leads to the conclusion
\begin{align*}
    &\Ebb\left[p_{k+1}^{-1}\left(F(\wv_{k+1})-F^*\right)|\Fc_k\right] \\  
    &\leq p_k^{-1}\left(F(\wv_k)-F^*\right) 
    - p_{k+1}^{-1}\alpha_k  \left(\mu \underline{\nu} - B \lambda \alpha_k \right)\|\nabla F(\wv_k)\|^2 + C \lambda (r_k -r_{k+1})~\mathrm{a.s.},   
\end{align*}
hence the result.
\end{proof}

\subsection{Choice of the Lyapunov function}
\label{subsec:Lya}

As highlighted in the previous subsection, we now aim to build a suitable Lyapunov function $V$ to apply Theorem~\ref{thKL:inequality}.
Considering the Euclidean space product $\Hc= \Rbb^N \times \Rbb \times \Rbb$ (of finite dimension), we deduce a family of processes able to verify Assumptions \ref{ass:H2} and \ref{ass:H4}.

\begin{proposition}\label{prop:sto_grad_Lya}
Assume that Assumptions~\ref{ass:H:5}\ref{ass:H:5:i}-\ref{ass:H:5:iv} hold, and assume that $\sup_{k\in \Nbb} \alpha_k < \mu \underline{\nu}(\lambda B)^{-1}$ (with $\lambda$ given in Lemma~\ref{lemma:grad}). 
Let $m\in \Nbb^*$ and define the process $(\xv_{m,k})_{k\in \Nbb}$ such that for every $k\in \Nbb$, $\xv_{m,k} = \left(\wv_k, p_k^{-1/(2m)}, r_k^{1/(2m)}\right) \in \Hc$, where $p_k$ and $r_k$ are defined in Corollary~\ref{cor:reform_descent}. Let 
\begin{equation}\label{eq:Lya_sto_grad}
\begin{array}{l@{}c@{}cl}
    V_m \colon & \Hc & \to & \Rbb \\
    & (\wb, t,r) & \mapsto & t^{2m}\left(F(\wb)-F^*\right) + C\lambda r^{2m}
\end{array}
\end{equation} 
and $\Gamma_m : = \zer \nabla F \times \left \{ p_{\infty}^{-1/(2m)}\right\} \times \{0\}$ where $p_{\infty}:=\prod\limits_{i=0}^{+\infty}(1+A \lambda \alpha_k^2)\in(0,+\infty)$. 
Then the following holds: 
\begin{enumerate}
    \item 
    The limit of $\left(F(\wv_k)\right)_{k\in \Nbb}$, denoted by $\FD_{\infty}$ exists and is integrable.
    \item $(\xv_{m,k})_{k\in \Nbb}$ satisfies Assumption \ref{ass:H2} and Assumption \ref{ass:H4} for $V=V_m$ and $\Gamma=\Gamma_m$ defined above.  
    \item 
    $\left(V_m\left(\xv_{m,k}\right)\right)_{k\in \Nbb}$ almost-surely converges to an integrable limit given by
    \begin{equation}\label{prop:sto_grad_Lya:proof:lya_lim}
     \VD_{m, \infty}= p_{\infty}^{-1/(2m)}\left(\FD_{\infty}-F^*\right)   .
    \end{equation}
    \item 
    If $\FD_{\infty}$ satisfies, for every $k\in \Nbb$, $\Ebb[\FD_{\infty}|\Fc_k]>\FD_{\infty}$ a.s., then Lemma~\ref{lemma:sum}\ref{lemma:sum:2} is satisfied.
\end{enumerate}

\end{proposition}

\vspace{0.4cm}

\begin{proof}
Let $m\in \Nbb^*$. 
\begin{enumerate}
\item Since Lemma~\ref{lemma:grad} holds with $\sum_{k=0}^{+\infty} \alpha_k^2<+\infty$ and Assumption \ref{ass:H:5}\ref{ass:H:5:vi} the integrable version of Robbins-Siegmund lemma can be used \cite{robbins1971convergence}. As such and almost-surely, $\FD_{\infty}$ is well-defined integrable.
\item 
In addition to the previous fact, we also have $\liminf_{k\to +\infty} \|\nabla F(\wv_k) \|=0$ due to $\sum_{k=0}^{+\infty} \alpha_k = +\infty$ and $\sup_{k\in \Nbb} \alpha_k < \mu \underline{\nu}(B\lambda)^{-1}$. 
We deduce that $\Upsilon_{\infty}$, the set of accumulation points of $(\wv_k)_{k\in \Nbb}$ almost-surely contains a point of $\zer \nabla F$. \\
Furthermore, since $r_k^{1/(2m)} \underset{k \to + \infty}{\to} 0$ and $p_k^{-1/(2m)} \underset{k \to + \infty}{\to} p_{\infty}^{-1/(2m)}$, then the set of accumulation points of $(\xv_{m,k})_{k\in \Nbb}$ is equal to $\chi_{m,\infty} = \Upsilon_{\infty} \times \left \{ p_{\infty}^{-1/(2m)} \right\} \times \{0\}$ and contains a point of $\zer \nabla F \left \{ p_{\infty}^{-1/(2m)} \right\} \times \{0\} $. 
According to Assumption~\ref{ass:H:5}\ref{ass:H:5:i}, $V_m$ is continuous and coercive.  
Hence, using Assumption \ref{ass:H:5}\ref{ass:H:5:ibis},
we have $|V_m(\zer \nabla F \times \{0\} \times \left \{ p_{\infty}^{-1}\right\}))|=|F(\zer \nabla F)|<+\infty$, 
and it follows that Assumption \ref{ass:H2}\ref{ass:H2:i} is satisfied. 

Due to assumption \ref{ass:H:5} \ref{ass:H:5:vi}, process $\left(V_m\left(\xv_{m,k}\right)\right)_{k\in \Nbb}$ is clearly integrable and, for all $k\in \Nbb$, \eqref{eq:rewrite_sto_grad} can be also rewritten as
\begin{equation}\label{prop:sto_grad_Lya:proof:1}
\Ebb\left[V_m\left(\xv_{m,k+1}\right)|\Fc_k\right] \leq V_m\left(\xv_{m,k}\right)
  - p_{k+1}^{-1}\alpha_k  \left(\mu \underline{\nu} - B\lambda \alpha_k \right)\|\nabla F(\wv_k)\|^2~~\mathrm{a.s..} 
\end{equation}
Finally, $\sup\limits_{k\in \Nbb} \alpha_k < \mu \underline{\nu}(\lambda B)^{-1}$ ensures for $\left(V_m\left(\xv_{m,k}\right)\right)_{k\in \Nbb}$ to be a supermartingale and so to verify Assumption~\ref{ass:H4}.

\item 
Following from (ii), $\left(V_m\left(\xv_{m,k}\right)\right)_{k\in \Nbb}$ almost-surely converges to an integrable limit $\VD_{m, \infty}$ and regarding (i) and \eqref{eq:rewrite_sto_grad}, the latter can be expressed as in \eqref{prop:sto_grad_Lya:proof:lya_lim}.

\item 
Then, $\FD_{\infty}$ is integrable and, almost-surely for all $k\in \Nbb$, due to $\Ebb[\FD_{\infty}|\Fc_k]>\FD_{\infty}$ we have $\Ebb[\VD_{m, \infty}|\Fc_k]>\VD_{m, \infty}$ and Lemma~\ref{lemma:sum}\ref{lemma:sum:2} holds. \end{enumerate}

\end{proof}

The previous proof relies on the convergence $(p_k)^{-1/(2m)}_{k\in \Nbb}$ $(r_k)_{k\in \Nbb}$ to a finite limit, and more specifically on the fact that their set of accumulation points is reduced to a singleton. This allows us to make easy links between process $(F(\wv_k))_{k\in \Nbb}$, directly associated with the initial scheme, and process $\left(V_m(\xv_{m,k})\right)_{k\in \Nbb}$ verifying the supermartingale condition. Asymptotically, the link between $\FD_{\infty}$ and $\VD_{m, \infty}$ is made throughout relation\eqref{prop:sto_grad_Lya:proof:lya_lim} whose remains moderate in term of complexity. 

In the case when the constant $A$ appearing in Assumption~\ref{ass:H:5}\ref{ass:H:5:iii} is equal to zero, the choice of the Lyapunov-type space and function can be simplified. Following the same technique of proof as for Proposition~\ref{prop:sto_grad_Lya} and using the fact that $p_k:=1$ in such a case, the following result is obtained.

\begin{proposition}\label{prop:sto_grad_Lya_bis}
Assume that Assumption~\ref{ass:H:5} holds with $A=0$, and assume that $\sup_{k\in \Nbb} \alpha_k < \mu \underline{\nu}(\lambda B)^{-1}$ (with $\lambda$ given in Lemma~\ref{lemma:grad}). 
Let $\check{\Hc} = \Rbb^N \times \Rbb$ and $m\in \Nbb^*$. Define the process $(\check{\xv}_{m,k})_{k\in \Nbb}$ such that for every $k\in \Nbb$, $\check{\xv}_{m,k} = \left(\wv_k, r_k^{1/(2m)}\right) \in \check{\Hc}$, where $r_k$ is defined in Corollary~\ref{cor:reform_descent}. Let 
\begin{equation}\label{prop:sto_grad_Lya_bis:lya_func}
\begin{array}{l@{}c@{}cl}
    \check{V}_m \colon & \check{\Hc} & \to & \Rbb \\
    & (\wb, r) & \mapsto & F(\wb)-F^* + C\lambda r^{2m}
\end{array}
\end{equation} 
and $\check{\Gamma} : = \zer \nabla F \times \{0\}$. 
Then the following holds: 
\begin{enumerate}
    \item The limit of $\left(F(\wv_k)\right)_{k\in \Nbb}$, denoted by $\FD_{\infty}$ exists and is integrable.
    \item $(\check{\xv}_{m,k})_{k\in \Nbb}$ satisfies Assumption \ref{ass:H2} and Assumption \ref{ass:H4} for $\check{V}_m$ and $\check{\Gamma}$ defined above. 
    \item 
    $\left(\check{V}_m\left(\xv_{m,k}\right)\right)_{k\in \Nbb}$ almost-surely converges to an integrable limit given by
    \begin{equation}\label{prop:sto_grad_Lya:proof:lya_lim_bis}
        \check{\VD}_{m, \infty}=F_{\infty} - F^*.
    \end{equation}
    \item 
    If $\FD_{\infty}$ satisfies, for every $k\in \Nbb$, $\Ebb[\FD_{\infty}|\Fc_k]>\FD_{\infty}$ a.s., then Lemma~\ref{lemma:sum}\ref{lemma:sum:2} is satisfied.
\end{enumerate}
\end{proposition}


\section{Application to generic SGD scheme}
\label{sec6}

In this section, we introduce our main asymptotical results related to scheme \eqref{eq:scheme:grad_sto}, ensuring the convergence in expectation for a sum close to $\sum_{k=0}^{+\infty} \alpha_k \| \nabla F(\xv_k)\|$ in a way to estimate the mean length of $(\wv_k)_{k\in \Nbb}$. To our knowledge, this result is novel, and although applied here to scheme \eqref{eq:scheme:grad_sto}, it could be mimicked to study the asymptotic behaviour of other stochastic processes. 
A simplified version of these results has been introduced by the authors in \cite{chouzenoux2023kurdyka}.

\subsection{Asymptotic behavior of generic SGD scheme in a non-convex setting}
\label{Ssec:sec6:gen}

We first aim to study the asymptotic behavior of the SGD scheme given in \eqref{eq:scheme:grad_sto}. Similarly to results given in subsection~\ref{subsec:Lya}, we distinguish the cases the constant $A$ of Assumption~\ref{ass:H:5:iii} is zero or not.

\begin{proposition} (Upper-bound in expectation for non-convex SGD) \label{prop:sum_sto_grad}
Assume that Assumption~\ref{ass:H:5} holds with $\sup_{k\in \Nbb} \alpha_k < \mu \underline{\nu}(\lambda B)^{-1}$ ($\lambda$ given in Lemma~\ref{lemma:grad}) and that $F$ verifies the K\L property on $\Rbb^N$, given in Definition~\ref{def_KL}. 
Then, according to Proposition~\ref{prop:sto_grad_Lya}, process $\left(F(\wv_k)\right)_{k\in \Nbb}$ almost-surely converges to an integrable limit $\FD_{\infty}$. Moreover, if the latter almost-surely verifies
\begin{equation}\label{prop:sum_sto_grad:cond}
        \displaystyle (\forall k \in \Nbb)\quad   
        \Ebb \big[ \FD_{\infty}|\Fc_k \big] > \FD_{\infty} 
        \text{ and }
       \|\nabla F(\wv_k)\|>0,
\end{equation}
then, for every $m\in \Nbb^*$, there exits a desingularization function $\varphi_m\in \Phi_{+\infty}$ and an almost-sure finite random variable $\KD_m>0$ such that
\begin{equation}\label{eq:prop:sum_sto_grad}
    \Ebb \left[ \sum_{k=\KD_m}^{+\infty} \frac{\alpha_k \|\nabla F(\wv_k)\|}{\sqrt{1 + 4m^2\YD_{m,k}}}\right] \leq \left(\frac{1+A\lambda \overline{\alpha}}{\underline{\nu} - B\lambda \overline{\alpha} }\right) \Ebb \big[ \varphi_m \big( V_m(\xv_{m,0})-\VD_{m,\infty} \big) \big],
\end{equation}
where $V_m$ (resp $\VD_{m,\infty}$) is the Lyapunov function (resp its limit) defined in \eqref{eq:Lya_sto_grad} (resp in \eqref{prop:sto_grad_Lya:proof:lya_lim}), and
\begin{equation}\label{eq:prop:sum_sto_grad:cond-defY}
    (\forall k \in \Nbb)\quad
    \YD_{m,k}=\frac{p_k^{1/m}(F(\wv_k)-F^*)^2+(C\lambda)^2 p_k^2 r_k^{2-1/m}}{\|\nabla F(\wv_k)\|^2}.
\end{equation}


\end{proposition}

\vspace{0.4cm}

\begin{proof} 
Let $m\in \Nbb^*$. 
The Lyapunov function $V_m$ defined in \eqref{eq:Lya_sto_grad} is differentiable (Assumption~\ref{ass:H:5}\ref{ass:H:5:i}), and 
\begin{equation}\label{prop:sum_sto_grad:proof:1}
(\forall (\wb,t, r) \in \Hc)\quad
    \nabla V_m(\wb,t, r) =
    \begin{pmatrix} t^{2m}\nabla F(\wb) \\ 2m t^{2m-1} \big( F(\wb) - F^* \big) \\ 2m C\lambda r^{2m-1}     \end{pmatrix}. 
\end{equation}
Moreover, since $F$ and $u\in \Rbb \mapsto u^{2m}$ verify the $\text{K\L}$   property, so it is for $V_m$. 
Then, according to Proposition~\ref{prop:sto_grad_Lya}, all the conditions for applying Theorem \ref{thKL} are met. 
Hence there exists $\varphi_m \in \Phi_{+\infty}$ as well as an almost-sure finite random variable $\KD_m>0$ such that
\begin{equation*}
    \Ebb\left[\sum_{k=\KD_m}^{+\infty} \frac{V_m(\xv_{m,k})-\Ebb \big[ V_m(\xv_{m,k+1})|\Fc_k \big]}{\|\nabla V_m(\xv_{m,k})\|_{\Hc}}\right] \leq \Ebb \big[ \varphi_m\big( V_m(\xv_{m,0}) - \VD_{m,\infty} \big) \big],
\end{equation*}
where $(\xv_{m,k})_{k\in \Nbb}$ is defined in Proposition~\ref{prop:sto_grad_Lya}.
Combining this inequality with \eqref{prop:sto_grad_Lya:proof:1} and \eqref{prop:sum_sto_grad:proof:1} with $(\wb, t, r) = (\wv_k, p_k^{-1/(2m)}, r_k^{1/(2m)})$, we obtain
\begin{multline*}
    \Ebb \left[\sum_{k=\KD_m}^{+\infty}  
    \frac{ p_{k+1}^{-1}\alpha_k \left(\underline{\nu} - B \lambda \alpha_k  \right) \| \nabla F(\wv_k)\|^2}%
        {\sqrt{p_k^{-2}\|\nabla F(\wv_k)\|^2 + 4m^2 \left(p_k^{1/m-2}(F(\wv_k)-F^*)^2 + (C\lambda)^2 r_k^{2-1/m}\right)}}\right] \\
    \leq \Ebb \big[ \varphi_m \big( V_m(\xv_{m,0})-\VD_{m,\infty} \big) \big].  
\end{multline*}
Factorizing by $\| \nabla F(\wv_k)\|^2$ into the square root of the denominator and using the fact that, for every $k\in \Nbb$, $p_{k+1}/p_k = (1+A\lambda \alpha_k)$, we obtain
\begin{multline*}
    \Ebb\left[\sum_{k=\KD_m}^{+\infty}  
    \frac{\alpha_k \left(\underline{\nu} - B \lambda \alpha_k  \right)(1+A\lambda \alpha_k)^{-1} \| \nabla F(\wv_k)\|}{\sqrt{1 + 4m^2 \left(p_k^{1/m}(F(\wv_k)-F^*)^2 + (C\lambda)^2 r_k^{2-1/m}\right)/\| \nabla F(\wv_k)\|^2}}\right]  \\
    \leq \Ebb[\varphi_m\left(V_m(\xv_{m,0})-\VD_{m,\infty}\right)].  
\end{multline*}
Consequently \eqref{eq:prop:sum_sto_grad} follows by using $\sup_{k\in \Nbb} \alpha_k < \underline{\nu}(\lambda B)^{-1}$. 
\end{proof}

Following the same line of proof, starting from Proposition~\ref{prop:sto_grad_Lya_bis}, we obtain a more specific upper-bound when $A=0$ in Assumption~\ref{ass:H:5}\ref{ass:H:5:iii}.  
\begin{proposition}\label{prop:sum_sto_grad_bis}(Upper-bound in expectation for non-convex SGD, $A=0$)
Assume that Assumption~\ref{ass:H:5} holds with $A=0$ and $\sup_{k\in \Nbb} \alpha_k < \mu \underline{\nu}(\lambda B)^{-1}$ ($\lambda$ given in Lemma~\ref{lemma:grad}) and that $F$ verifies the $\text{K\L}$   property on $\Rbb^N$, given in Definition~\ref{def_KL}. 
Then, according to Proposition~\ref{prop:sto_grad_Lya_bis}, process $\left(F(\wv_k)\right)_{k\in \Nbb}$ almost-surely converges to an integrable limit $\FD_{\infty}$. Moreover, if \eqref{prop:sum_sto_grad:cond} holds, 
then, for every $m\in \Nbb^*$, there exits a desingularization function $\check{\varphi}_m\in \Phi_{+\infty}$ and an almost-sure finite random variable $\check{\KD}_m$ such that
\begin{equation}\label{eq:prop:sum_sto_grad_bis}
    \Ebb \left[ \sum_{k=\check{\KD}_m}^{+\infty} \frac{\alpha_k \|\nabla F(\wv_k)\|}{\sqrt{1 + 4m^2\check{\YD}_{m,k}}}\right] \leq \left(\underline{\nu} - B\lambda \overline{\alpha} \right)^{-1} \Ebb \big[ \check{\varphi}_m \big( \check{V}_m(\xv_{m,0})-\check{\VD}_{m,\infty} \big) \big],
\end{equation}
where $\check{V}_m$ is the Lyapunov function defined in \eqref{prop:sto_grad_Lya_bis:lya_func}, $\check{\VD}_{m,\infty}$ the limit given in \eqref{prop:sto_grad_Lya:proof:lya_lim_bis}, and
\begin{equation*}
    (\forall k \in \Nbb)\quad
    \check{\YD}_{m,k} 
    = \frac{(C\lambda)^2 r_k^{2-1/m}}{\|\nabla F(\wv_k)\|^2}. 
\end{equation*}


\end{proposition}


\subsection{Particular cases for stronger convergence results}
\label{Ssec:sec6:special}

The results in Propositions~\ref{prop:sum_sto_grad} and \ref{prop:sum_sto_grad_bis} give upper-bounds on the expectation of the summation of a power of the norm of the gradient of $F$.
In this section we will investigate particular cases allowing to deduce stronger convergence guarantees. In particular we first focus on the Strong Growth Condition, and then explore another type of assumption.

\subsubsection{Almost-sure convergence of SDG under Strong Growth Condition} 

We consider the process generated in~\eqref{eq:scheme:grad_sto}, and focus on the case  when $A=C=0$ in Assumption~\ref{ass:H:5}\ref{ass:H:5:iii}, i.e., the usual Strong Growth Condition \cite{schmidt2013fast}.
In this case, it can be noticed that, for every $m\in \Nbb^*$, $\check{\YD}_{m,k}=0$, hence the following result is a direct application of Proposition~\ref{prop:sum_sto_grad_bis}.


\begin{proposition}(Almost sure convergence of SGD under Strong Growth Condition)\label{cor:cvunderSGC}
Let $(\wv_k)_{k\in \Nbb}$ be the process generated by~\eqref{eq:scheme:grad_sto}.
Under the same assumptions as in Proposition~\ref{prop:sum_sto_grad_bis}, and assuming further that $C=0$ in Assumption~\ref{ass:H:5}\ref{ass:H:5:iii}, the process $(\wv_k)_{k\in \Nbb}$ almost surely converges to a stationary point of $F$.
Moreover, there exits $\varphi\in \Phi_{+\infty}$ and an almost-surely finite random variable $\KD>0$ such that
\begin{equation}\label{eq:cor:cvunderSGC}
    \Ebb \left[ \sum_{k=\KD}^{+\infty} \Ebb \big[ \|\wv_{k+1}-\wv_k\|~|~\Fc_k \big] \right]
    \leq \frac{1}{\overline{\nu} \sqrt{B}} \left(\underline{\nu} - \overline{\alpha} B \lambda \right)^{-1} \Ebb \big[ \varphi \big(F(\wv_0)-\FD_{\infty} \big) \big],
\end{equation}
where $\FD_{\infty}$ is the integrable limit of $\left(F(\wv_k)\right)_{k\in \Nbb}$. 
\end{proposition}

\vspace{0.4cm}

\begin{proof}
Since $C=0$, we have, for every $k\in \Nbb$ and $m\in \Nbb^*$, $\check{\YD}_{m,k}=0$ in Proposition~\ref{prop:sum_sto_grad_bis}, and by definition $\check{V}_m(\xv_{m,k})=F(\wv_k) - F^*$ and $\check{\VD}_{m, \infty}=F_{\infty} - F^*$ (see \eqref{prop:sto_grad_Lya_bis:lya_func} and \eqref{prop:sto_grad_Lya:proof:lya_lim_bis}, respectively). It follows that \eqref{eq:prop:sum_sto_grad_bis} can be simplified as:
\begin{equation}\label{proof:cor:cvunderSGC:1}
\Ebb \left[ \sum_{k=\check{\KD}_m}^{+\infty} \alpha_k \|\nabla F(\wv_k)\|\right] 
\leq \left(\underline{\nu} - \overline{\alpha}B\lambda \right)^{-1} \Ebb \big[ \varphi_m \big( F(\wv_0)   -\FD_{\infty}\big) \big].
\end{equation}
Let $k\in \Nbb$. According to \eqref{eq:scheme:grad_sto}, and using Assumption~\ref{ass:H:5}\ref{ass:H:5:iii}, we have
\begin{align*}
    \Ebb\left[\|\wv_{k+1}-\wv_k\|~|~\Fc_k\right] 
    = \alpha_k \left[\|\Bv_k \fv_k\|~|~\Fc_k\right] 
    \leq \alpha_k \overline{\nu}~ \Ebb \left[\| \fv_k\|~|~\Fc_k\right] .
\end{align*}
Then, applying Jensen's inequality to the last inequality, combined with Assumption~\ref{ass:H:5}\ref{ass:H:5:iv}, yield
\begin{equation*}
    \Ebb\left[\|\wv_{k+1}-\wv_k\|~|~\Fc_k\right] 
    \leq \alpha_k \overline{\nu}~ \|\Ebb \left[ \fv_k~|~\Fc_k\right]\|
    \leq \alpha_k \overline{\nu} \sqrt{B} \|\nabla F(\wv_k)\|.
\end{equation*}
Combining the last result with \eqref{proof:cor:cvunderSGC:1}, and taking $\varphi=\check{\varphi}_m$ and $\KD=\check{\KD}_m$, we obtain \eqref{eq:cor:cvunderSGC}.

\smallbreak

Since $\KD>0$ is finite almost-surely and that all terms in \eqref{eq:cor:cvunderSGC} are positive, we can deduce that
\begin{equation*}
    \sum_{k=0}^{+\infty} \Ebb\left[\|\wv_{k+1}-\wv_k\|~|~\Fc_k\right]<+\infty~~\mathrm{a.s..}
\end{equation*}
Since, fro every $k\in \Nbb$, $\|\wv_{k+1}-\wv_k\|\geq 0$, we can use
Levy’s sharpening of Borel-Cantelli Lemma \cite[ Ch.1, Th.21]{meyer2006martingales} leading to
\begin{equation*}
\sum_{k=0}^{+\infty}\|\wv_{k+1}-\wv_k\|<+\infty~~\mathrm{a.s.}.
\end{equation*}
Almost-surely $(\wv_k)_{k\in \Nbb}$ is then of finite length (i.e. Cauchy) and so it converges. Finally, similarly to the proof of Proposition~\ref{prop:sto_grad_Lya}, $(\wv_k)_{k\in \Nbb}$ also possesses (almost-surely) an accumulation point contained in $\zer \nabla F$, hence the conclusion.  
\end{proof}

\subsubsection{Convergence guarantees under particular convergence rate conditions}

In Proposition~\ref{cor:cvunderSGC} we shown that, under the assumption that $A=C=0$, the process $(\wv_k)_{k\in \Nbb}$ is of finite length almost surely.
In this section we aim to extend this result for a class of events even when $A$ and $C$ are non-zero.
Hence we introduce the two following events:
\begin{align}
\label{event_Xi_1}
    &\Xi_1 := \bigcup\limits_{m\in \Nbb^*} \left\{ \limsup\limits_{k\to +\infty}\left[ 
    \dfrac{\left(\sum_{i=k}^{+\infty}\alpha_i^2\right)^{2-\frac{1}{m}}}{ \| \nabla F(\wv_k)\|^{2}}\right]<+\infty\right\} ,  \\
    \label{event_Xi_2}
    &\Xi_2 := \left\{ \limsup\limits_{k\to +\infty} \frac{\left(F(\wv_k)-F^*\right)}{\| \nabla F(\wv_k)\|}<+\infty\right\},
\end{align}
that will act as \textit{control bounds} on the behaviour of the gradient and the step-size of SGD. Before giving further intuition on these bounds, we give the following convergence guarantees on these events.

\begin{corollary}\label{th:cv_sto_grad_1}
    We consider the same assumptions as in Proposition~\ref{prop:sum_sto_grad}.
    Then 
    \begin{equation*}
        \sum_{k=0}^{+\infty}\alpha_k \| \nabla F(\wv_k)\|<+\infty 
    \end{equation*}
    almost-surely on the event 
    \begin{equation}
    \label{th:cv_sto_grad_1_Xi}
    \Xi := 
    \begin{cases}
      \Xi_1~\text{if}~A=0, \\
      \Xi_2 ~\text{if}~C=0, \\
      \Xi_1 \cap \Xi_2 ~\text{otherwise}.
    \end{cases}
    \end{equation}
\end{corollary}

\vspace{0.4cm}

\begin{proof}
We first look at the case $\Xi=\Xi_1 \cap \Xi_2~$.
We introduce    
\begin{equation*}
\Ac  :=  \bigcap_{m\in \Nbb^*} \left\{ \sum_{k=0}^{+\infty} \frac{\alpha_k \|\nabla F(\wv_k)\|}{\sqrt{1 + \YD_{m,k}}} < + \infty \right\}.
\end{equation*}
Let $\omega \in \Xi\cap \Ac $. By definition of $\Xi$ and $\Xi_1$ more specifically, there exits $m_{\omega}>0$ s.t. 
\begin{equation}
    \dfrac{\left(\sum_{i=k}^{+\infty}\alpha_i^2\right)^{2-\frac{1}{m_{\omega}}}}{ \| \nabla F(\wv_k(\omega))\|^{2}}<+\infty. 
\end{equation}
Then, reformulating \eqref{eq:prop:sum_sto_grad:cond-defY} we have
\begin{equation*}
    \YD_{m_\omega,k}(\omega)
    = p_k^{1/m_\omega}\dfrac{(F(\wv_k(\omega))-F^*)}{\|\nabla F(\wv_k(\omega))\|}
    + (C\lambda)^2 p_k^2 \dfrac{r_k^{2-1/m_\omega}}{\|\nabla F(\wv_k(\omega))\|^2}.
\end{equation*}
On the one hand, since $(p_k)_{k\in \Nbb}$ is bounded (as it converges to $p_\infty$)
On the other hand, since, by definition of $(r_k)_{k\in \Nbb}$ (see \eqref{eq:def:rk-pk}), we have $r_k < \sum_{i=k}^{+\infty} \alpha_i^2$ then the second term in the above equation is also finite. It follows that $\left(\YD_{m_{\omega},k}(\omega)\right)_{k\in \Nbb}$ is bounded and so
\begin{equation}\label{bigO}
    \frac{1}{\sqrt{1 + \YD_{m_{\omega},k}(\omega)}}\underset{k \to +\infty}{=}\Oc(1). 
\end{equation}

As $\omega \in \Ac$
\begin{equation}\label{th:cv_sto_grad_1:proof:2}
\sum_{k=0}^{+\infty}\frac{\alpha_k \|\nabla F(\wv_k(\omega))\|}{\sqrt{1 + \YD_{m_{\omega},k}(\omega)}} <+\infty, 
\end{equation}
and, from positivity $\left(\YD_{m_{\omega},k}(\omega)\right)_{k\in \Nbb}$ and \eqref{bigO}, this leads to
\begin{equation}\label{th:cv_sto_grad_1:proof:3}
\sum_{k=0}^{+\infty}\alpha_k \|\nabla F(\wv_k(\omega))\|<+\infty. 
\end{equation}
Since Proposition \ref{prop:sum_sto_grad} ensures that $\Pbb(\Ac)=1$, \eqref{th:cv_sto_grad_1:proof:3} thus occurs for almost every $\omega$ of $\Xi$, hence the conclusion. 

Cases $\Xi = \Xi_1$ and $\Xi = \Xi_2$ can be treated in a similar way (consider $\left(\check{Y}_{m,k}\right)_{k\in \Nbb}$ instead of $\left(Y_{m,k}\right)_{k\in \Nbb}$ for the case $A=0$).
 
\end{proof}

The practical use of Corollary~\ref{th:cv_sto_grad_1} depends on the probabilities $\Pbb(\Xi_1), \Pbb(\Xi_2)$, that should be close to $1$ to be of interest. We discuss below a few examples where this happens to be the case ($\Xi_2$ first, $\Xi_1$ then).

\begin{example}\label{ex:cv-SGD-cvx}\ 
    If $F$ is assumed to be $\varrho$-strongly convex ($\varrho >0$), then the following particular version 
    of \L ojasiewicz inequality (see e.g. \cite{karimi2016linear}) holds
    \begin{equation} \label{PL-inequality}
        \forall w\in \Rbb^N~~\| \nabla F(w) \|^2 \geq 2 \varrho \left(F(w)-F^*\right)
    \end{equation}
    which ensures that 
    \begin{equation} \label{PL-inequality_application}
    \forall k \in \Nbb \quad \frac{F(\wv_k) - F^*}{\| \nabla F(\wv_k)\|} = \frac{\sqrt{F(\wv_k) - F^*}}{\| \nabla F(\wv_k)\|} \sqrt{F(\wv_k) - F^*} \leq \frac{1}{\sqrt{2\varrho}} \sqrt{F(\wv_k) - F^*}. 
    \end{equation}
    Since $\left(F(\wv_k) - F^*\right)_{k\in \Nbb}$ almost-surely converges to the almost-sure finite limit $F_{\infty}-F^* \geq 0$, \eqref{PL-inequality_application} conducts to $\Pbb(\Xi_2)=1$ with
    \begin{equation}
        \limsup_{k\to +\infty}\frac{F(\wv_k) - F^*}{\| \nabla F(\wv_k)\|} \leq \frac{1}{\sqrt{2\varrho}} \sqrt{F_{\infty} - F^*}~~\mathrm{a.s..}
    \end{equation}    
\end{example}

To investigate more general cases than the convex settings described in Example~\ref{ex:cv-SGD-cvx}, we need to take a closer look at the convergence rates of the SGD scheme.
Intuitively, as discussed in \cite{sebbouh2021almost}, the gradient $(\|\nabla F(\wv_k)\|)_{k\in \Nbb}$ should not converge too fast to $0$ to enable the event $\Xi$ to not be too small, as shown in the following example. 

\begin{example}\label{ex:cv-SGD}
Let, for every $k\in \Nbb$, $\alpha_k=\frac{a}{k^{\frac12+\epsilon}}$ with $a>0$ small enough and $\epsilon\in (0,\frac12)$. Then, 
according to \cite[Corollary 18]{sebbouh2021almost}, we have
\begin{equation*}
\min_{1\leq i\leq k}\| \nabla F(\wv_i)\|^2 \underset{k\to +\infty}{=}o\left(\frac{1}{k^{\frac12-\epsilon}} \right)~~\mathrm{a.s.}. 
\end{equation*}
Hence, intuitively, we can postulate that $\| \nabla F(\wv_i)\|^2$ does not decrease faster than $\frac{1}{k^{1/2}}$. 
As such, assuming that $\| \nabla F(\wv_k)\|^2\underset{k\to +\infty}{\sim} \frac{1}{k^{1/2}}$ almost surely, then for all $m\in \Nbb^*$, we deduce that  
\begin{equation*}
    \dfrac{\left(\sum_{i=k}^{+\infty}\alpha_i^2\right)^{2-\frac{1}{m}}}{ \| \nabla F(\wv_k)\|^2} 
    \underset{k\to +\infty}{\sim} \left(\frac{a^2}{k^{2\epsilon}}\right)^{2-\frac{1}{m}} k^{\frac12} 
    \underset{k\to +\infty}{=} \Oc \left(k^{\left(\frac12-2\epsilon(2-\frac{1}{m})\right)} \right)~~\mathrm{a.s.}. 
\end{equation*}  
Since $\frac12-2\epsilon(2-\frac{1}{m}) \underset{m \to +\infty}{\longrightarrow}\frac12-4\epsilon$, we deduce that, if $\epsilon \in [\frac{1}{8}, \frac12)$, there exists $m_0\in \Nbb^*$ for which $\frac12-2\epsilon(2-\frac{1}{m_0})\leq 0$ and
\begin{equation*}
    \dfrac{\left(\sum_{i=k}^{+\infty}\alpha_i^2\right)^{2-\frac{1}{m_0}}}{ \| \nabla F(\wv_k)\|^2}\underset{k\to +\infty}{=} \Oc(1)~~\mathrm{a.s.}
\end{equation*}
Hence we have $\Pbb(\Xi_1)=1$.  
\end{example}

\section{Discussion and conclusion}
\label{sec7}

In this work, we presented a new framework to investigate convergence properties of stochastic schemes in a non-convex context. 

In Sections~\ref{sec3}-\ref{sec4},  we introduced a framework to apply the KL property in a stochastic setting. 
In this context the notion of ``decay", that is usually key to study convergence on deterministic schemes in a non-convex setting, is associated with that of supermartingale. It then appears to be more difficult to satisfy in the stochastic than in the deterministic case. To circumvent this difficulty, we made the choice to add mild conditions as those described in Theorem \ref{thKL}. 
This allowed us to derive new asymptotic results leveraging mathematical tools different from the standard ODE method and ``trajectory-by-trajectory" strategy used in \cite{benaim1996dynamical,dereich2021convergence, tadic2015convergence}. 
In Sections~\ref{sec5}-\ref{sec6}, we studied the case of the SGD scheme to show that our assumptions are reasonable in practice. In particular, they do not require specific curvature properties for the functions on interest. 

\smallskip

Finally, we focused on the particular case of smooth functions. However, as highlighted in \cite{bolte2010characterizations,Bolte2014}, we are naturally led to consider a non-smooth extension of our framework. 
Using non-smooth version of the $\text{K\L}$ property, Sections~\ref{sec3}-\ref{sec4} could be straightforwardly extended to non-smooth coercive and continuous Lyapunov functions $V$.
The difference would only be of structural order, by replacing the gradient norm by the distance of the limiting subdifferential \cite{mordukhovich2006variational} to zero. 
Note that the stochastic non-smooth case has been investigated for particular schemes (see e.g., \cite{davis2016asynchronous} or \cite{hertrich2022inertial}). 




\bibliographystyle{siamplain}
\bibliography{references}

\end{document}